\documentclass{article}
\usepackage{amsthm}
\usepackage[utf8]{inputenc}
\usepackage[english]{babel}
\usepackage{amsmath,amsfonts,amssymb, amsthm}
\usepackage[shortlabels]{enumitem}
\usepackage{amssymb,amsfonts}
\usepackage[all,arc]{xy}
\usepackage{mathrsfs}
\usepackage{tikz-cd}
\usepackage{mathtools}
\usepackage{parskip}
\usepackage{cancel}
\usepackage{url}
\usepackage{biblatex}
\usepackage{csquotes}

\usepackage{graphicx}

\usepackage{lipsum}
\usepackage{hyperref}
\usepackage{cleveref}

\newcommand\myshade{85}
\colorlet{mylinkcolor}{violet}
\colorlet{mycitecolor}{blue}
\colorlet{myurlcolor}{orange}

\hypersetup{
  linkcolor  = mylinkcolor!\myshade!black,
  citecolor  = mycitecolor!\myshade!black,
  urlcolor   = myurlcolor!\myshade!black,
  colorlinks = true,
}

\newtheorem{theorem}{Theorem}[section]
\newtheorem{remark}[theorem]{Remark}
\newtheorem{prop}[theorem]{Proposition}
\newtheorem{corollary}[theorem]{Corollary}
\newtheorem{lemma}[theorem]{Lemma}

\theoremstyle{definition}

\newcommand{\Z}{\mathbb{Z}}

\DeclareSymbolFont{bbold}{U}{bbold}{m}{n}
\DeclareSymbolFontAlphabet{\mathbbold}{bbold}

\begin{document}

\thispagestyle{plain}
\begin{center}
    \Large
    \textbf{Arcs Intersecting at Most Once on the 4-Punctured Sphere}
        
    \vspace{0.4cm}
    \large
    \vspace{0.4cm}
    \textbf{Paul Tee}
       
    \vspace{0.9cm}
    \textbf{Abstract}
\end{center}
We classify maximal systems of arcs which intersect at most once on the 4-punctured sphere.

\tableofcontents
\newpage

\section{Acknowledgements}
First and foremost, I would like to thank my advisor Piotr Przytycki for not only supporting me academically but also for helping me adjust to a new life in a difficult time. Thank you to Denali Relles for noticing an extra maximal 1-system in an earlier version of this paper. Thank you to my wonderful girlfriend who has been my rock throughout these two years. Thank you to parents who have been there for me throughout this journey of life. Lastly, thank you to my friends, especially to those in and around the reading room. 

\section{Introduction}
In this paper, we classify maximal systems of arcs which intersect at most once on the 4-punctured sphere. We begin with a survey of the relevant definitions following \cite{smith2017arcs}.

Let $S$ be a surface of finite type with punctures and Euler characteristic $\chi<0$. In certain natural situations, we choose to call punctures \textit{vertices}. An \textit{arc} on $S$ is a map from $(0, 1)$ to $S$ that is proper. A proper map induces a map between topological ends of spaces, and in this sense each endpoint of $(0, 1)$ is sent to a puncture of $S$. We will say that the arc is \textit{between} these punctures. An arc is \textit{simple} if it is an embedding. In that case we can and will identify the arc with its image in $S$. An arc $\alpha$ is \textit{essential} if it cannot be homotoped into a puncture in the sense that there is no proper map $(0,1)\times [0,1]\to S$ which, restricted to $(0,1) \times 0$, equals $\alpha$. All arcs in this paper are simple and essential. An arc is a \textit{loop} if it starts and ends at the same puncture. In this case, we say the loop is \textit{based} at the puncture. A \textit{homotopy} between arcs $\alpha$ and $\beta$ is a proper map $(0,1)\times [0,1]\to S$ which, restricted to $(0, 1) \times 0$ equals $\alpha$, and restricted to $(0, 1) \times 1$ equals $\beta$. In particular, $\alpha$ and $\beta$ start at the same puncture and end at the same puncture. 

We say that arcs $\alpha$ and $\beta$ are in minimal position, if the number of intersection
points $|\alpha\cap \beta|$ cannot be decreased by a homotopy. A \textit{bigon} (respectively, \textit{half-bigon}) between arcs $\alpha$ and $\beta$ is an embedded closed disc $B\subset S$ (respectively,
properly embedded half-disc $B=[0, 1] \times [0, 1) \subset S$) such that $B\cap (\alpha\cup \beta) = \partial B$
and both $\partial B\cap \alpha$ and $\partial B\cap \beta$ are connected. It is a well known fact, we call the \textit{bigon criterion for arcs}, that $\alpha$ and $\beta$ are in minimal position if and only if they are transverse and there is no bigon or half-bigon between them.

We define a \textit{$k$-system} $\mathbb{A}$ on a surface $S$ to be a collection of pairwise nonhomotopic arcs such that any two arcs intersect at most $k$ times. Note that a $0$-system is a set of disjoint arcs. We define \textit{k}-systems $\mathbb{A},\mathbb{B}$ to be \textit{equivalent} if for every arc $\alpha\in \mathbb{A}$ there exists an arc $\beta\in \mathbb{B}$ such that $\alpha\simeq \beta$, where $\simeq$ denotes homotopy. We denote this by $\mathbb{A}\simeq \mathbb{B}$. Define the \textit{degree} of $\mathbb{A}$ to be $\deg(\mathbb{A}):= (\deg p_i)_{i\in \mathcal{P}}$ where $\mathcal{P}$ denotes the set up punctures equipped with a total order. If $\mathbb{A}\simeq \mathbb{B}$, then up to reordering, $\deg(\mathbb{A})=\deg(\mathbb{B})$.

We say $\mathbb{A}$ is  \textit{saturated} if for any arc $\sigma$, $\{\sigma\}\cup \mathbb{A}$ is no longer a $k$-system. Theorem 1.5 of     \cite{p2015} implies for a fixed surface $S$, the cardinality of any saturated $k$-system is bounded by a constant that depends only on $S$ and $k$. Thus, we say a saturated $\mathbb{A}$ is a \textit{maximal} $k$-system if for any saturated $\mathbb{B}$, we have 
$$|\mathbb{B} | \leq |\mathbb{A}|.$$ 

Consider a $k$-system $\mathbb{A}$ where the arcs are pairwise in minimal position. There is a canonical subset, the \textit{disjoint subset},
$\mathbb{J}\subset \mathbb{A}$ consisting of arcs which are disjoint all other arcs in $\mathbb{A}$. In other words, 
$$\mathbb{J}:=\{\sigma\in \mathbb{A}| \sigma \cap \tau = \emptyset,  \forall \tau\in \mathbb{A}-\{\sigma\}\}$$

\begin{remark}\label{max0sys}
Any saturated 0-system on $S$ has cardinality $3|\chi|$.
\begin{proof}
Choose a hyperbolic metric for $S$, and then by the  Gauss--Bonnet Theorem, the area of $S$ is $2\pi|\chi|$. Any saturated system forms an ideal triangulation of $S$. As any ideal triangle has area $\pi$, we have $2|\chi|$ triangles, so $6|\chi|$ edges. Every edge is counted twice in a triangulation, so we have $3|\chi|$ arcs.
\end{proof}
\end{remark}

The following provides an analogous statement of Remark \ref{max0sys} for 1-systems.

\begin{theorem}\label{max1sys}
\cite[Theorem 1.2]{p2015} Any maximal 1-system on $S$ has cardinality $2|\chi|(|\chi|+1)$.
\end{theorem}

Throughout, denote $\Sigma$ as the four-punctured sphere, and let $\overline{\Sigma}$ denote the compactification of $\Sigma$ by points labelled by $\mathcal{P}=(a,b,c,d)$ corresponding to the punctures. We denote our working representation of $\Sigma$ (equivalently, $\overline{\Sigma}$) as the 3-punctured disk $D$ with boundary $d$, identified as shown in Figure \ref{fig:workingrep}. Then, each arc ending (or starting) at $d$ can
be homotoped to an arc that converges to a point on $\partial D$, and we will be only
considering such arcs. Note that homotopic arcs ending at $d$ might converge to
different points of $\partial D$.

\begin{figure}[htp]
    \centering
    \includegraphics[width=4cm]{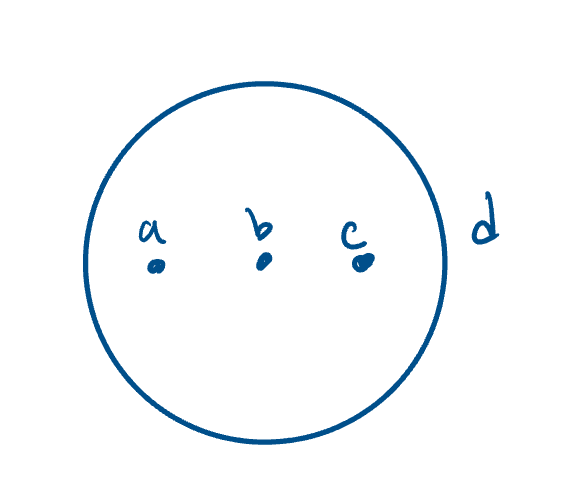}
    \caption{$\Sigma$ as the 3-punctured open disk}
    \label{fig:workingrep}
\end{figure}

Using this representation, we state our two main results. Our classification is taken with respect to self-homeomorphisms of $\Sigma$ and equivalence of $\mathbb{A}$. 
\begin{lemma}
There are six saturated 0-systems on $\Sigma$, as show in Figure \ref{fig:max0}.
\begin{figure}[htp]
    \centering
    \includegraphics[width=7cm]{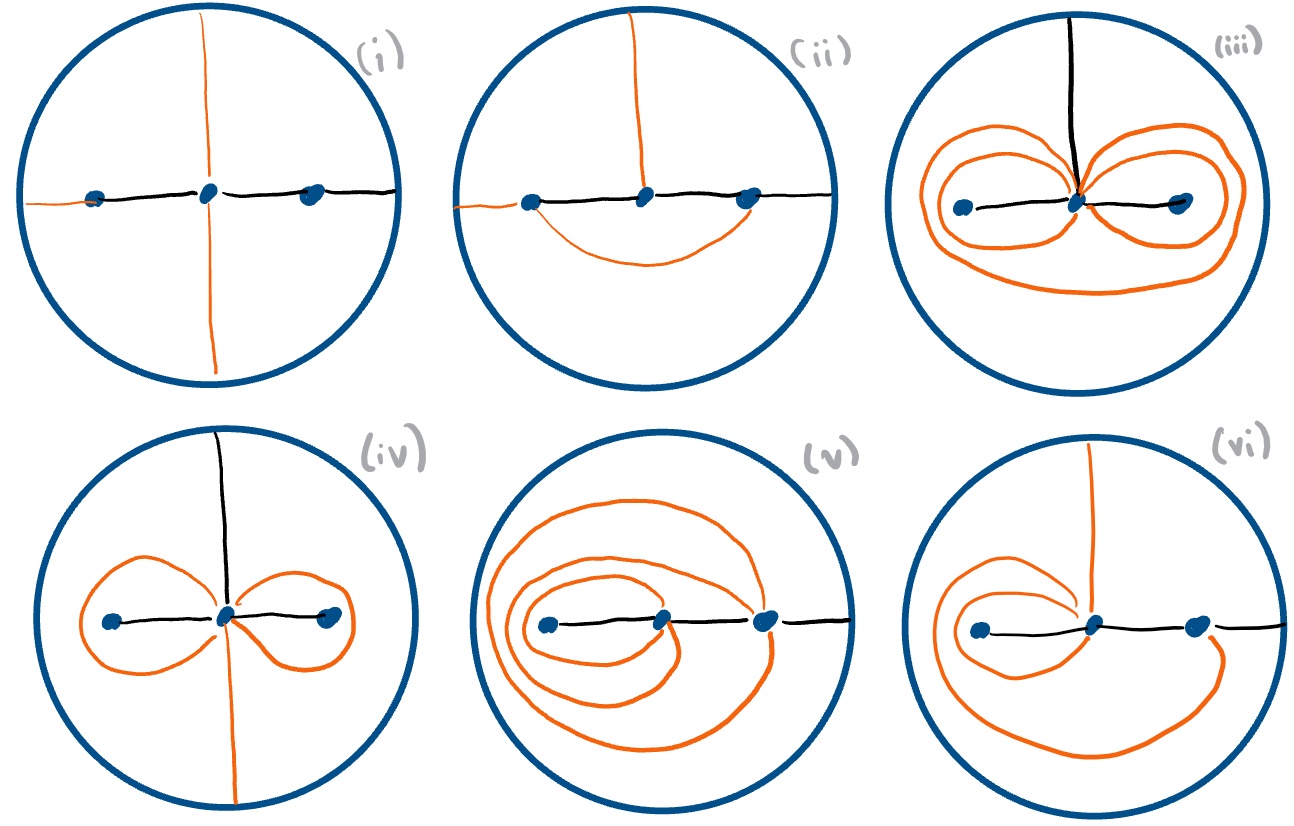}
    \caption{Saturated 0-systems labelled (i)-(vi)}
    \label{fig:max0}
\end{figure}
\end{lemma}

\begin{theorem}
There are nine maximal 1-systems on $\Sigma$, as shown in Figure \ref{allmax1}.

\begin{figure}[htp]
    \centering
    \includegraphics[width=12cm]{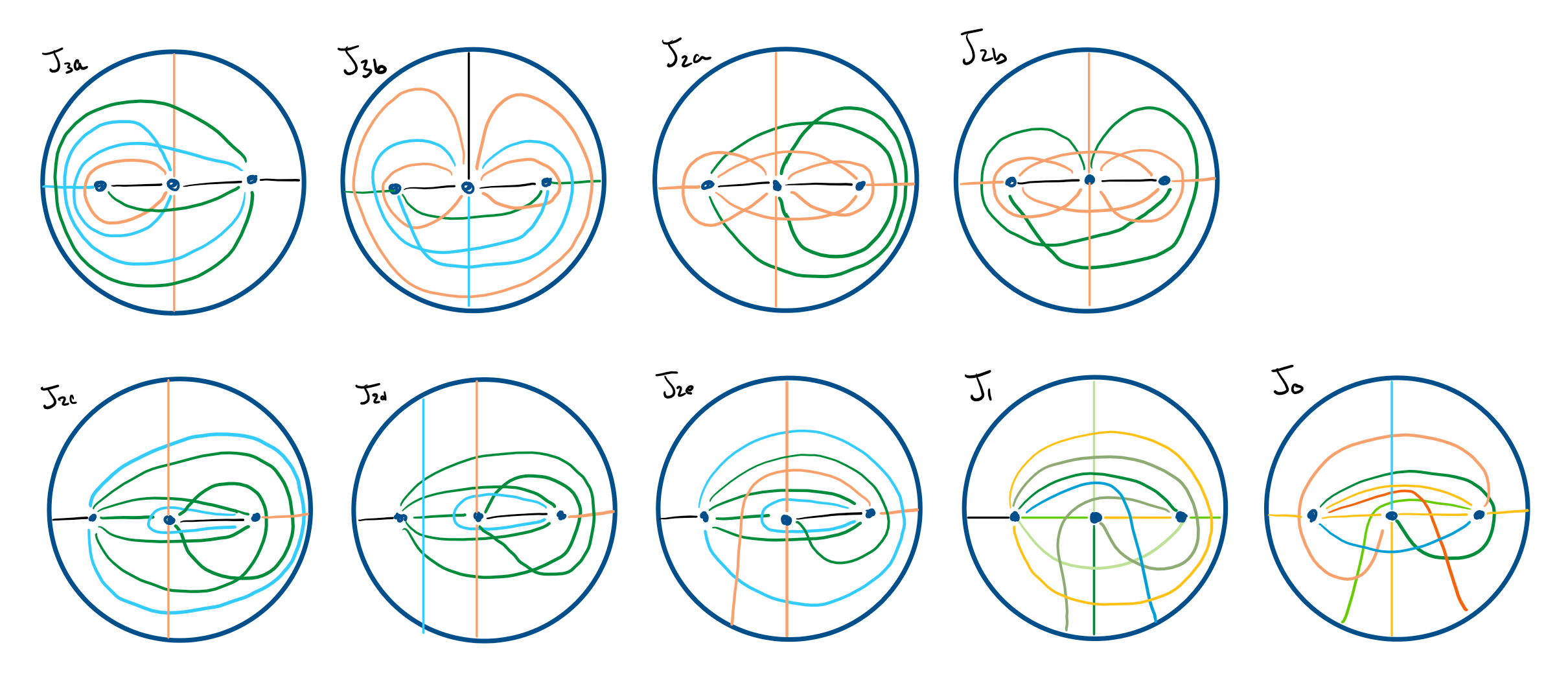}
    \caption{Maximal 1-systems labelled $J_{3a}-J_0$}
    \label{allmax1}
\end{figure}
\end{theorem}
Note that for each maximal 1-system the black arcs are exactly $\mathbb{J}$. The naming convention is as follows: the number denotes the number of arcs in $|\mathbb{J}|$, and for a fixed $|\mathbb{J}|$, the letters (if they exist) index the distinct maximal 1-systems.

\textbf{Organization: }In Section \ref{0systems} (resp. Section \ref{1systems}), we  use Remark \ref{max0sys} (resp. Theorem \ref{max1sys}) to classify saturated 0-systems (resp. maximal 1-systems) on $\Sigma$. In Section \ref{0systems}, we first show any saturated 0-system must contain an embedded tree, and the two subsections deal with each case separately. In Section \ref{deets}, we discuss several technical lemmas we need through the rest of the paper, before culminating in showing $|\mathbb{J}|\leq 3$. This section can be used as reference material for Section \ref{1systems}. In Section \ref{1systems}, the subsections \ref{3cut}, \ref{2cut}, \ref{1cut}, \ref{0cut} describe the $|\mathbb{J}|=$ 3, 2, 1 and 0 cases respectively.

\section{Classifying Saturated 0-Systems}\label{0systems}
Let $\mathbb{A}$ be a saturated $0$-system on $\Sigma$. As $\chi(\Sigma)=-2$, we have $|\mathbb{A}|=6$ by Remark \ref{max0sys}. We begin by observing that $\mathbb{A}$ defines a embedded graph $\Gamma=\Gamma_\mathbb{A}\subset \overline{\Sigma}$ with vertex set $\mathcal{P}$ and edge set $\mathbb{A}$.

\begin{lemma}\label{Gamma_connected}
If $\mathbb{A}$ is saturated and $\Gamma$ is as defined above, then $\Gamma$ is connected.
\begin{proof}
Indeed, suppose $\Gamma$ is not connected. Take any arc $\sigma\subset \overline{\Sigma}-\Gamma$ between any two connected components which misses $\Gamma$ except at endpoints. Note that $\{\sigma\}\cup \mathbb{A}$ defines a $0$-system in $\Sigma$, giving the desired contradiction.
\end{proof}
\end{lemma}
By Lemma \ref{Gamma_connected}, we can find a tree $T\subset \Gamma$ with four vertices. As there are two isomorphism types of maximal trees on four vertices, the two trees $T_a,T_b$ in Figure \ref{fig:embtree} illustrate the two different ways of embedding a tree into $\overline{\Sigma}$, and this embedding is unique up to a homeomorphism of $\overline{\Sigma}$.
\begin{figure}[htp]
    \centering
    \includegraphics[width=6cm]{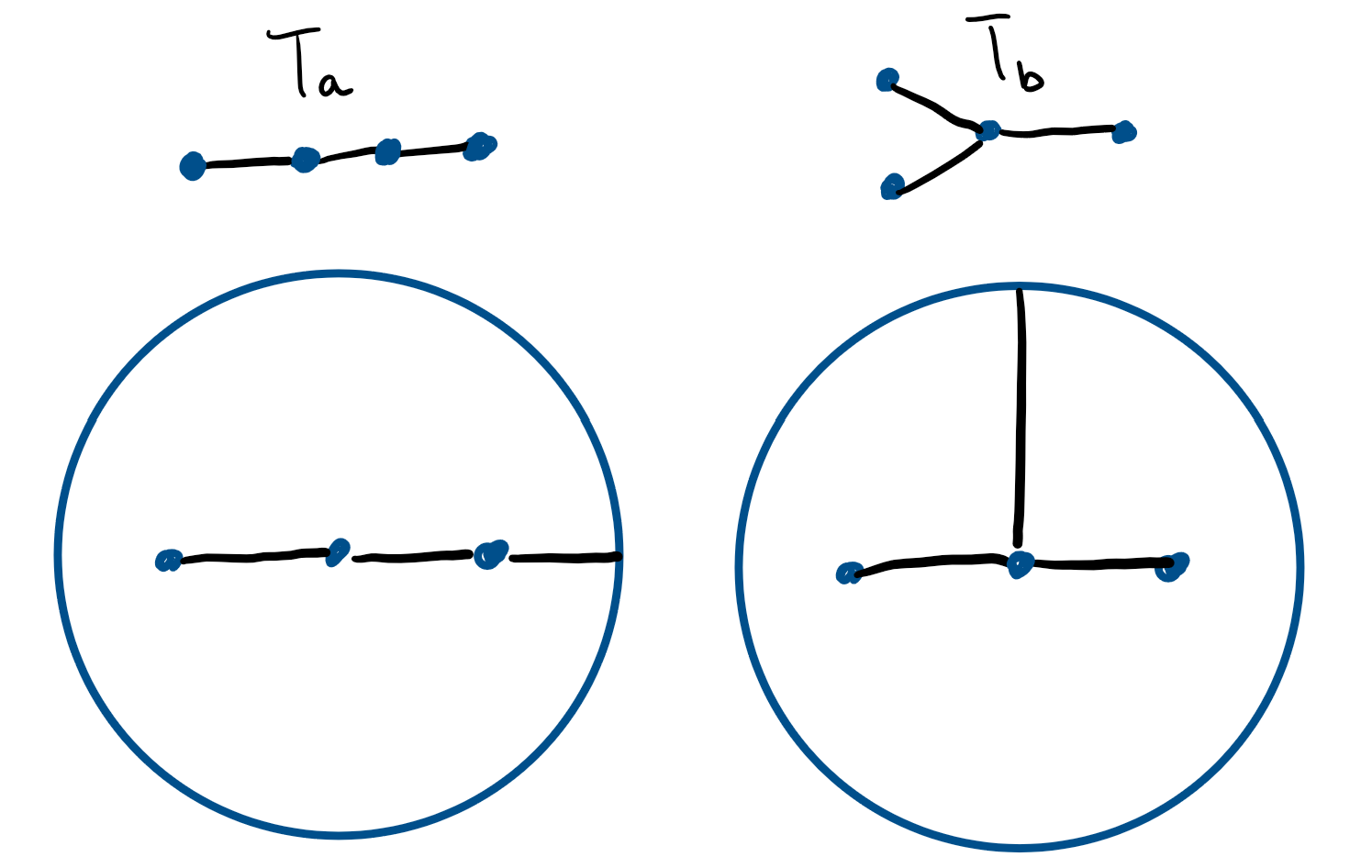}
    \caption{Embedded maximal trees $T_a,T_b$}
    \label{fig:embtree}
\end{figure}

We complete $T$ to a saturated 0-system if we can find three disjoint, nonhomotopic arcs in $\overline{\Sigma}-T$. Begin by observing that cutting along $T$ gives an ideal hexagon in either case. To complete $T$ to a saturated 0-system, we have three configurations of arcs as seen in Figure \ref{fig:hexagons}. The remaining work is to put the labels corresponding to $T_a,T_b$ into the hexagon. 

\begin{figure}[htp]
    \centering
    \includegraphics[width=3cm]{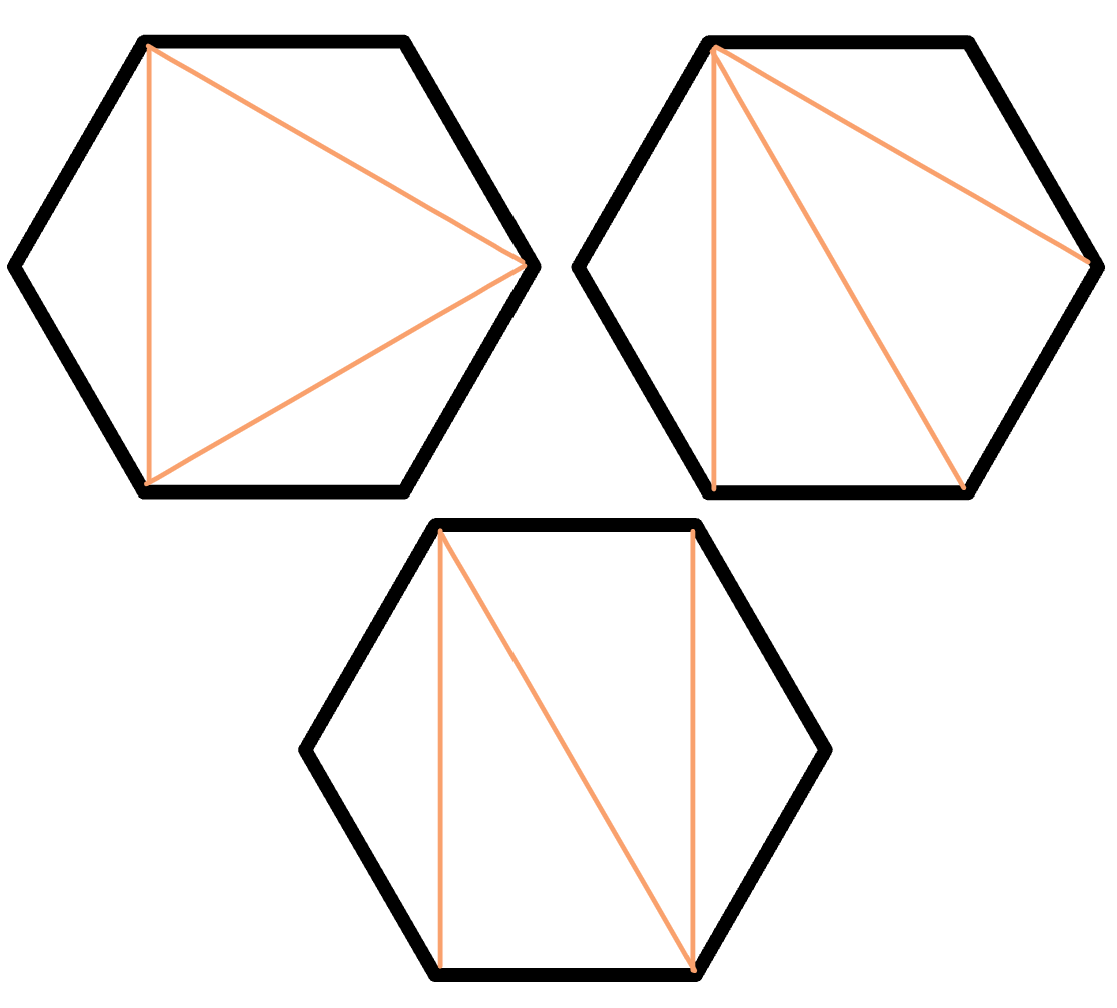}
    \caption{$\Delta,\epsilon, $ N configurations}
    \label{fig:hexagons}
\end{figure}
\subsection{Tree \texorpdfstring{$T_a$}{Lg}}
Note there is an action of $\Z/2$ on this graph, which interchanges $a$ with $d$ and $b$ with $c$. We can extend elements of $\Z/2$ to homeomorphisms of $\overline{\Sigma}$ which preserve~$T_a$. We classify the configurations of Figure~$\ref{fig:hexagons}$ up to these homeomorphisms.

\begin{enumerate}
    \item [$\Delta$:] Up to an action of $\Z/2$, there is one $\Delta$ configuration, illustrated in Figure \ref{fig:T1Delta}.
    
    \begin{figure}[htp]
    \centering
    \includegraphics[width=2.5cm]{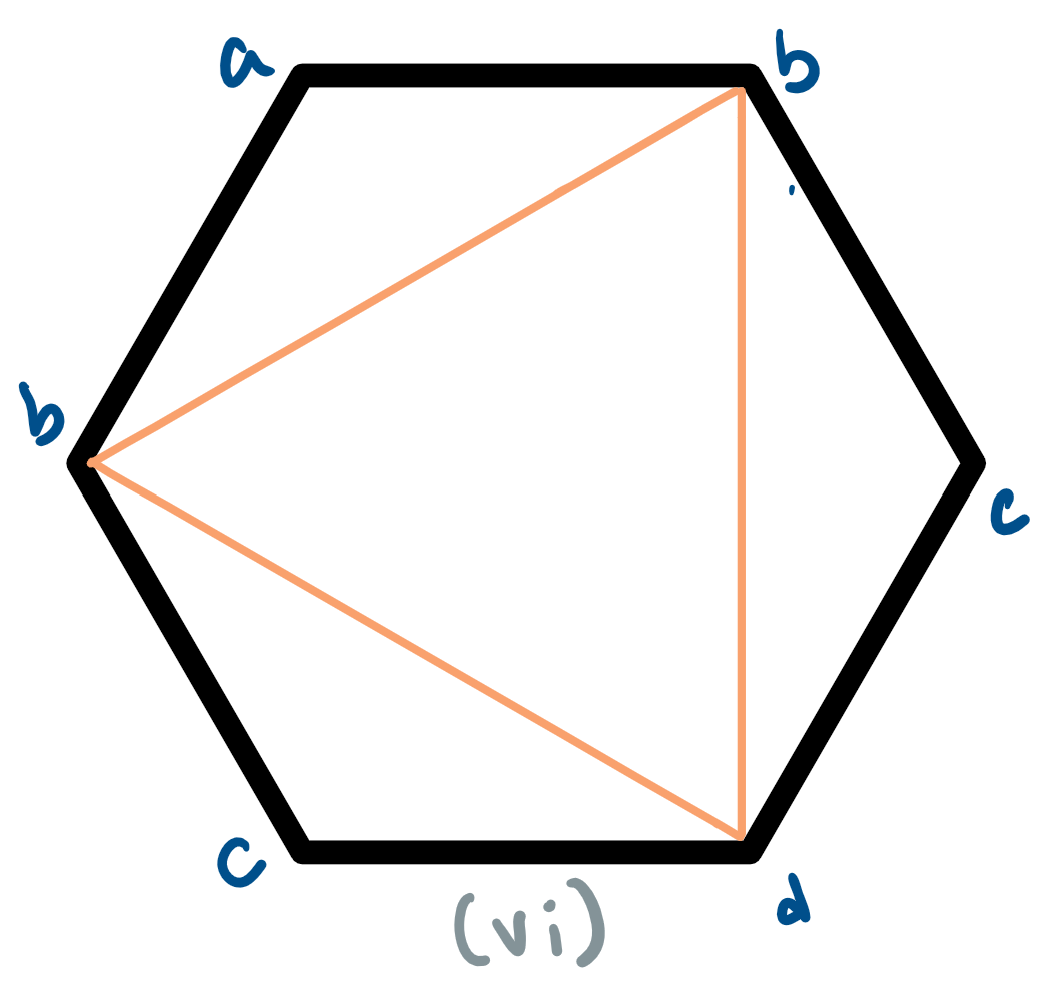}
    \caption{Homeomorphic to (vi) in Figure \ref{fig:max0}}
    \label{fig:T1Delta}
    \end{figure}
    
    \item [$\epsilon$:] Up to an action of $\Z/2$, there  are two $\epsilon$ configurations, illustrated in Figure \ref{fig:T1Epsilon}.
    
    \begin{figure}[htp]
    \centering
    \includegraphics[width=4.5cm]{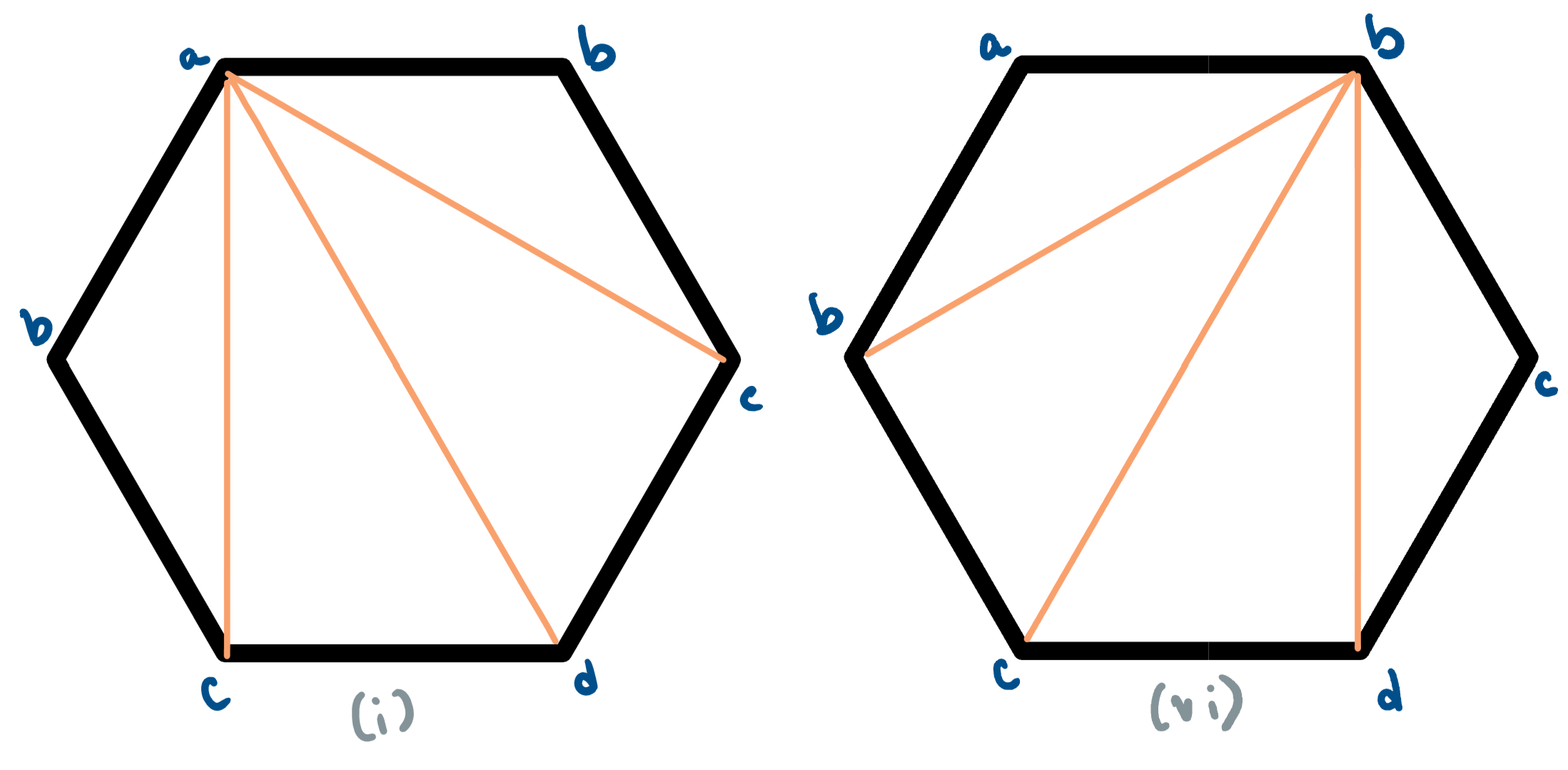}
    \caption{Homeomorphic to (i), (vi) in Figure \ref{fig:max0}}
    \label{fig:T1Epsilon}
    \end{figure}
    
    \item [N:] Up to an action of $\Z/2$, there are three $N$ configurations, illustrated in Figure \ref{fig:T1N}.
    
    \begin{figure}
    \centering
    \includegraphics[width=4cm]{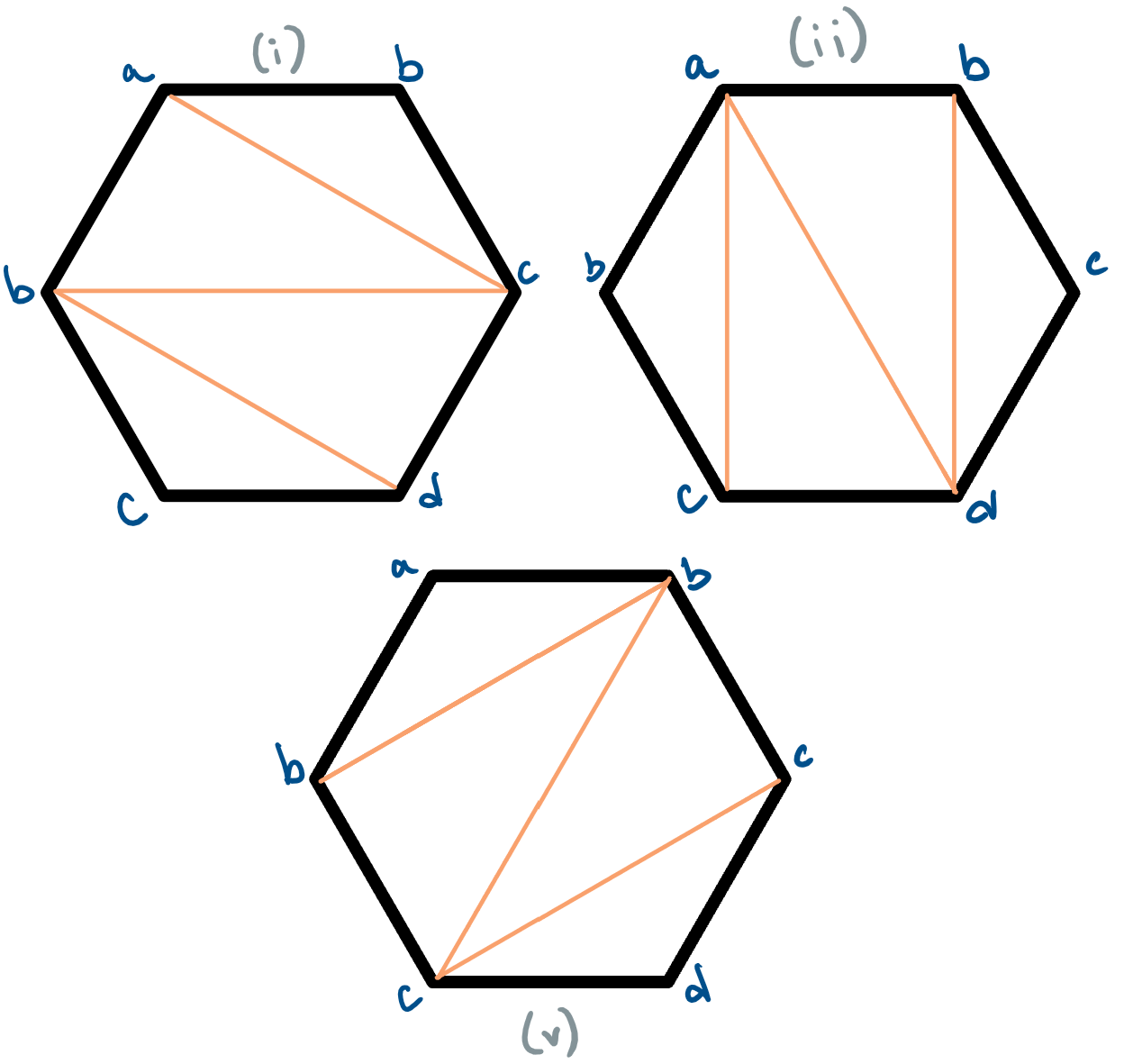}
    \caption{Homeomorphic to (i), (ii), (v) in Figure \ref{fig:max0}}
    \label{fig:T1N}
    \end{figure}
    
\end{enumerate}
\subsection{Tree \texorpdfstring{$T_b$}{Lg}}
Note that for $V=\{a,c,d\}$, we have an action of the symmetry group $S_V$ on this graph. We can extend elements of $S_V$ to homeomorphisms of $\overline{\Sigma}$ which preserve $T_b$. We classify the configurations of Figure $\ref{fig:hexagons}$ up to these homeomorphisms.

\begin{enumerate}
    \item [$\Delta$:] Up to an action of $S_V$, there are two $\Delta$ configurations, illustrated in Figure \ref{fig:T2Delta}.
    
    \begin{figure}[htp]
    \centering
    \includegraphics[width=4.5cm]{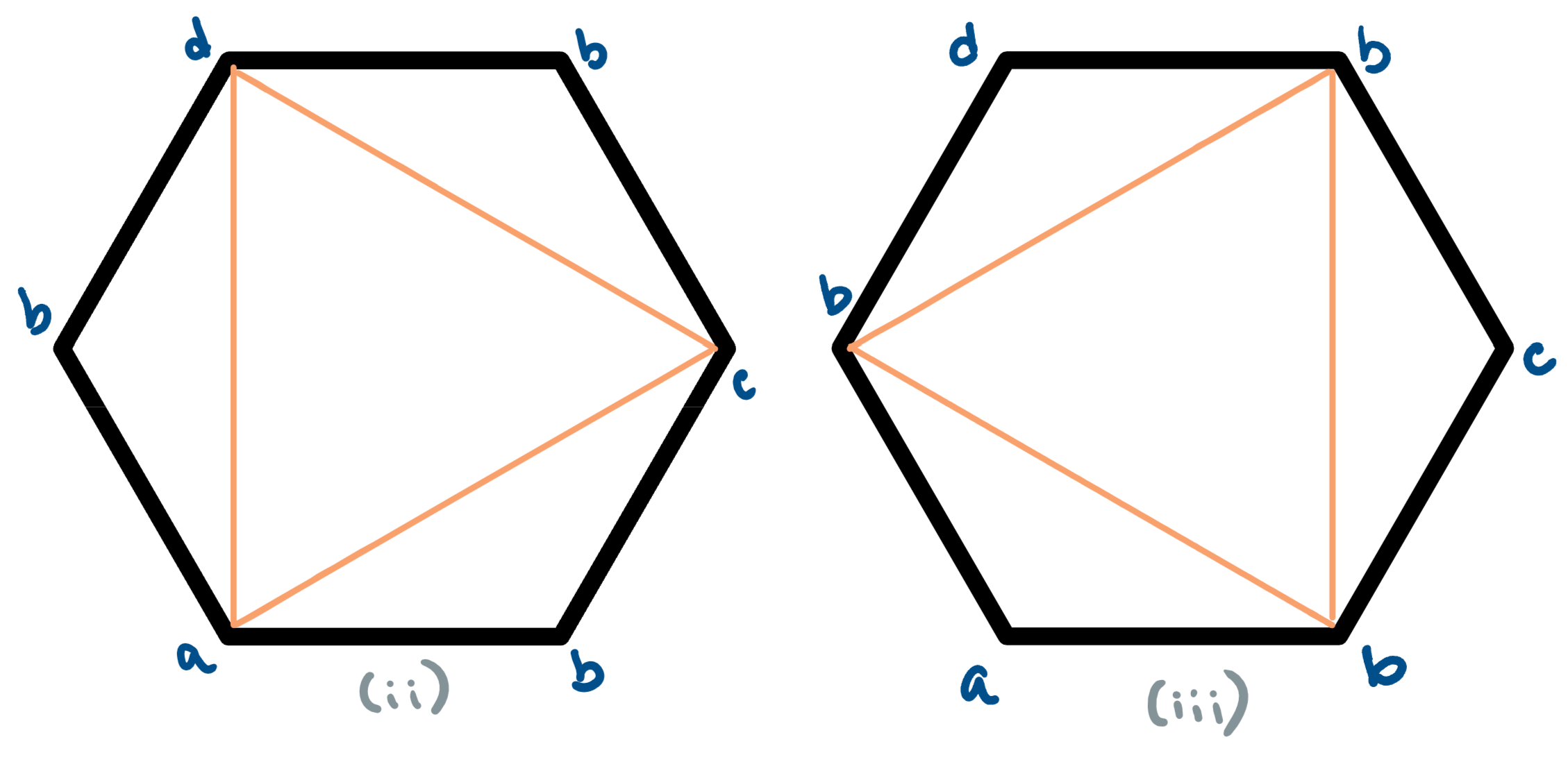}
    \caption{Homeomorphic to (ii), (iii) in Figure \ref{fig:max0}}
    \label{fig:T2Delta}
    \end{figure}
    
    \item [$\epsilon$:] Up to an action of $S_V$, there are two $\epsilon$ configurations, illustrated in Figure \ref{fig:T2Epsilon}.
    
    \begin{figure}[htp]
    \centering
    \includegraphics[width=4.5cm]{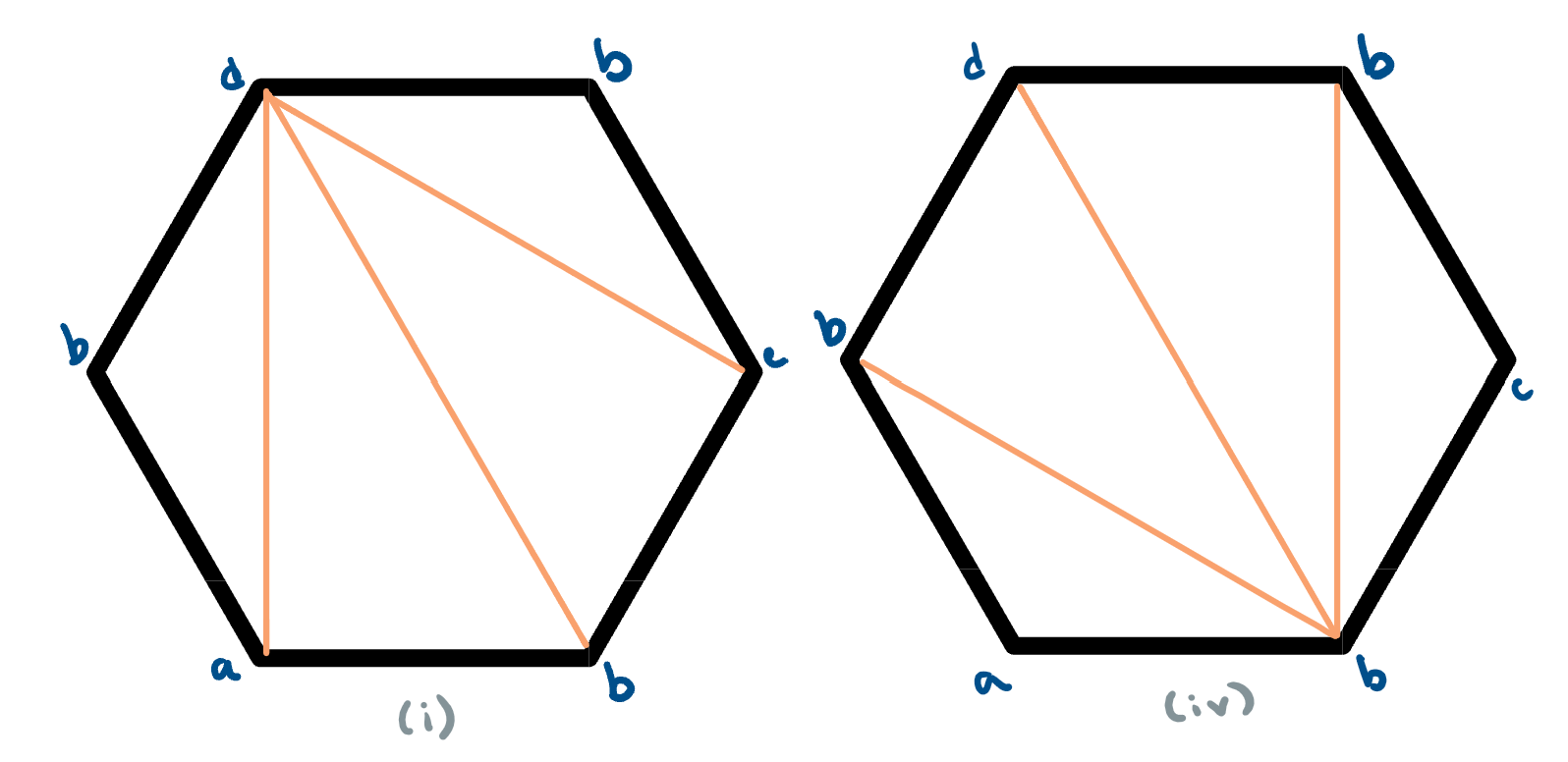}
    \caption{Homeomorphic to (iv), (i) in Figure \ref{fig:max0}}
    \label{fig:T2Epsilon}
    \end{figure}
    
    \item [N:] Up to an action of $S_V$, there is one $N$ configuration, illustrated in Figure \ref{fig:T2N}.
    
    \begin{figure}[htp]
    \centering
    \includegraphics[width=2.5cm]{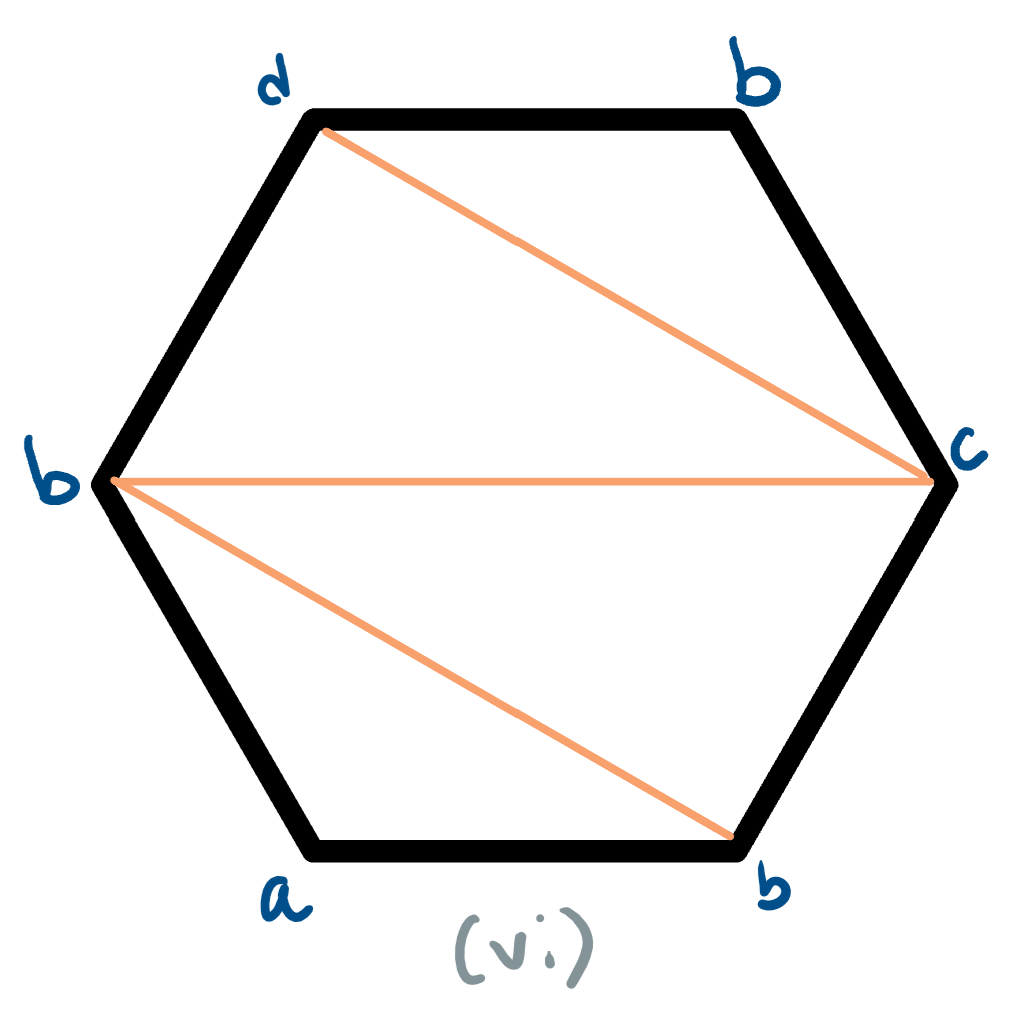}
    \caption{Homeomorphic to (vi) in Figure \ref{fig:max0}}
    \label{fig:T2N}
    \end{figure}
    
\end{enumerate}

We can distinguish systems (i)-(vi) by consideration of the set of the vertex degrees of $\Gamma$, which is easily seen to be an invariant of $\mathbb{A}$. Thus, all saturated 0-systems on $\Sigma$ are given by Figure \ref{fig:max0} up to homeomorphisms and equivalence.

\section{Properties of 1-systems}\label{deets}
Let $\mathbb{A}$ be a maximal $1$-system on $\Sigma$. As $\chi(\Sigma)=-2$, we have $|\mathbb{A}|=12$ by Theorem \ref{max1sys}. 

\begin{lemma}\label{loopisajail}
Consider a loop $\ell$ based at puncture $p$ which encloses a once-punctured disk $D_{pq}$ with puncture $q$. Let $\sigma_{pq}\subset D_{pq}$ be an arc between $p$ and $q$. For any maximal 1-system $\mathbb{A}$ such that $\ell\in \mathbb{A}$, then necessarily $\sigma_{pq}\in \mathbb{J}$.  
\begin{proof}
For any $\alpha\in \mathbb{A}$, $\alpha$ must satisfy exactly one of the two conditions: 
\begin{enumerate}
    \item $\alpha$ can be homotoped to an arc inside $\Sigma-D_{pq}$.
    \item $\alpha$ has an endpoint at $q$ and connected intersection with $D_{pq}$. 
\end{enumerate}
These two conditions follow from the bigon criterion for arcs. The lemma follows as each arc in condition 2 is homotopic to or disjoint from $\sigma_{pq}$.
\end{proof}
\end{lemma}

\begin{lemma}\label{tech}
Let $\mathcal{D}_n$ be a once-punctured $n$-gon for $n=1,2,3,4$ with puncture $p_{-\infty}$, and $\mathbb{A}=\mathbb{A}_n$ a maximal 1-system on $\mathcal{D}_n$. Then $\mathcal{D}_n$ contains at most $1, 3, 6, 10$ arcs respectively.
\begin{proof}
The 1,2-gon cases are obvious. For $\mathcal{D}_3$, note that there are 3 loops and 6 arcs which are not loops - 3 joining vertices and 3 joining a vertex with a puncture. 

\begin{figure}[htp]
    \centering
    \includegraphics[width=3cm]{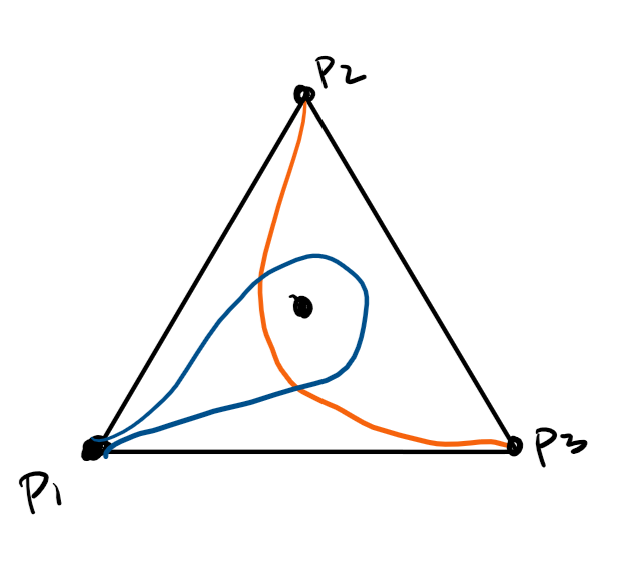}
    \caption{Correspondence between arcs joining distinct vertices and loops}
    \label{punctri}
\end{figure}

However, not all configurations of arcs are possible. In Figure \ref{punctri}, the blue loop excludes the following orange arc, and this type of restriction is the only possible one. Consequently, there are at most 6 arcs in $\mathbb{A}$. In fact, see Figure \ref{punctrimax1} for an example.

\begin{figure}[htp]
    \centering
    \includegraphics[width=3cm]{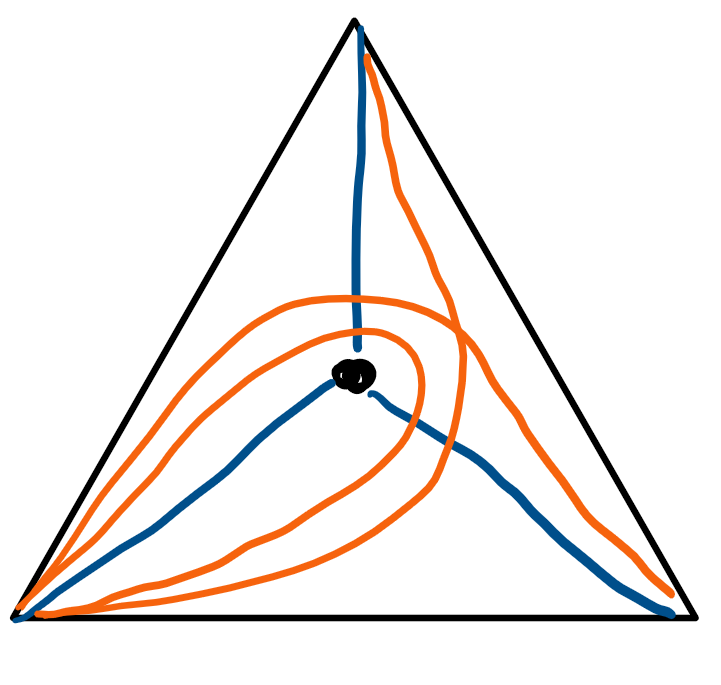}
    \caption{A maximal 1-system on $\mathcal{D}_3$}
    \label{punctrimax1}
\end{figure}

For $\mathcal{D}_4$, we count 4 loops, and 12 arcs which are not loops. This comes from
4 arcs joining a vertex to the puncture, 4 arcs joining adjacent vertices which we call \emph{adjacent arcs}, and 4 arcs which join opposite vertices. The restriction in $\mathcal{D}_3$ generalizes to the following statement: Suppose first there is a loop $\ell \in \mathbb{A}$. Apart from other loops, we have exactly three arcs, illustrated in Figure \ref{puncsq}, which intersect l more than once and hence are excluded from $\mathbb{A}$. This shows $|\mathbb{A}|\leq 10$. Observe also, we can have at most two adjacent arcs in any 1-system, and these two adjacent arcs must share a common endpoint. 

\begin{figure}[htp]
    \centering
    \includegraphics[width=3cm]{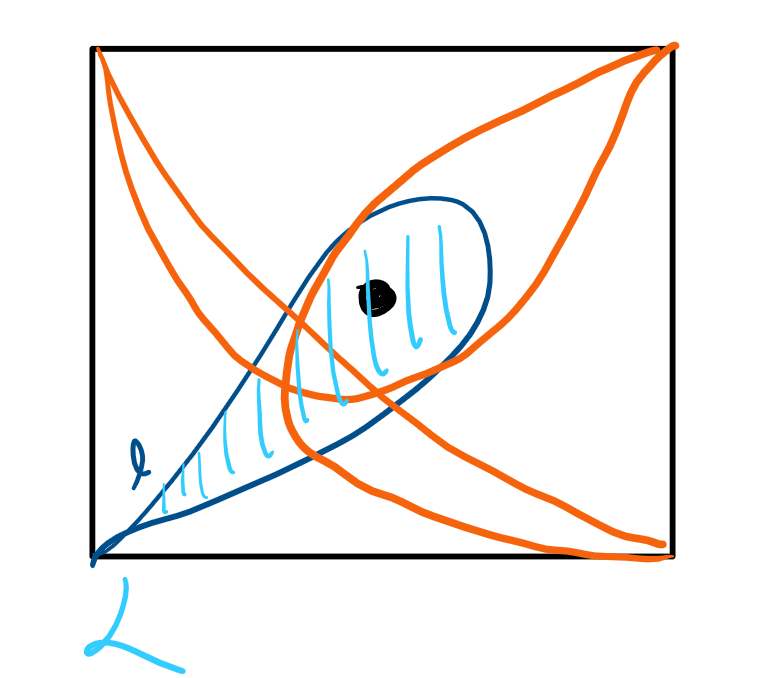}
    \caption{Arcs excluded by a loop}
    \label{puncsq}
\end{figure}
If we have no loops, it is clear the constraints discussed above force any 1-system to contain at most ten arcs. 
\end{proof}
\end{lemma}

Finally, consider $\Omega=S^1\times [0,1]$, where each boundary component is given a cell decomposition of two vertices and two edges. In the following discussion, we only consider arcs which are not loops. Label $S^1\times \{0\}$ vertices $a,d$ and $S^1\times \{1\}$ vertices $b,c$. Let $\alpha_0=\{a\}\times [0,1]$ be the ``vertical arc" between $a,b$. For $n\in \Z$, denote $\alpha_n$ as image of $\alpha_0$ under $n$ Dehn twists about $S^1\times \{\frac{1}{2}\}$. Observe
$$|\alpha_n\cap \alpha_m|=|n-m|-1$$ Similarly, define $\delta_0=\{d\}\times [0,1]$, and $\delta_n$ the image of $\delta_0$ under $n$ Dehn twists about $S^1\times \{\frac{1}{2}\}$, and a similar observation applies. 

We generalize this observation. Let $X$ denote the set of homotopy classes of arcs between the two boundary components which start and end at a vertex. Thus, $X=X_a^b\sqcup X_a^c\sqcup X_d^b \sqcup X_d^c$, where $X_a^b$ denotes the set of homotopy classes of arcs between $a$ and $b$, and so on. Define a map
$$\tau:X_a^b\sqcup X_d^c\to \Z$$
$$\tau:\alpha_n, \delta_n \mapsto n$$
For $\sigma\in X_a^c\sqcup X_d^b$, choose a representative which lies in the disk bounded by $\alpha_n\cup \alpha_{n+1}$ and  $\delta_n\cup \delta_{n+1}$ for a unique $n\in \Z$. 

\begin{figure}[htp]
    \centering
    \includegraphics[width=7cm]{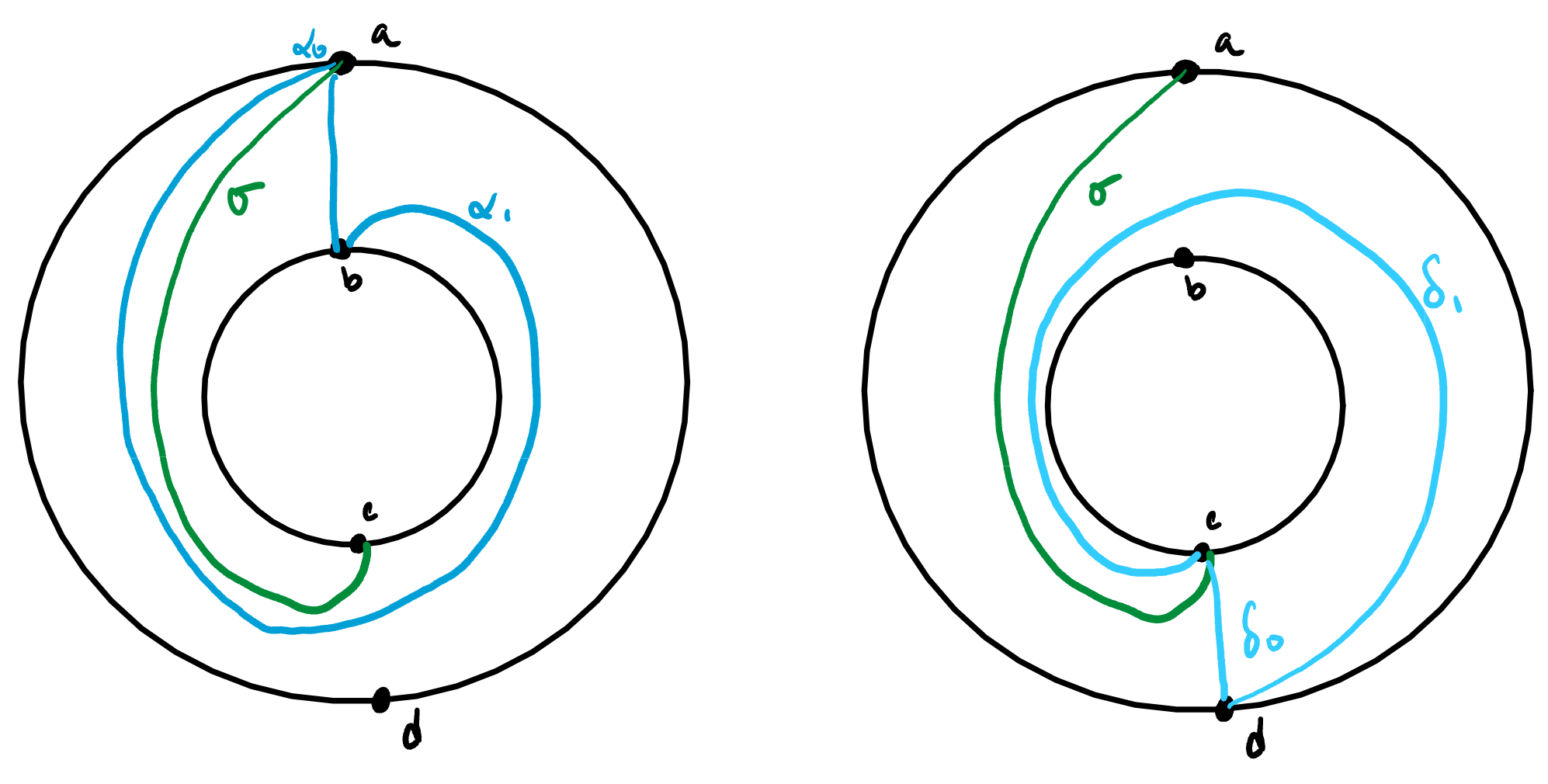}
    \caption{$\sigma$ in the region bounded by  $\alpha_0,\alpha_1$ (left) and $\delta_0$, $\delta_1$ (right)}
    \label{twex}
\end{figure}

Then the formula
$$\tau(\sigma)=n+\frac{1}{2}$$
defines a map $\tau:X\to \frac{1}{2}\Z$ which we call the \textit{twisting number}. By induction, it is easy to prove that $\tau$ computes the number of intersections of two arcs in minimal position: 

\begin{lemma}\label{mainth} Let $\sigma_1,\sigma_2\in X$,  and let $\eta$ denote the number of common endpoints between $\sigma_1$ and $\sigma_2$. Then,
$$|\sigma_1\cap \sigma_2|=|\tau(\sigma_1)-\tau(\sigma_2)|-\frac{\eta}{2}$$
\end{lemma}

\begin{figure}
    \centering
    \includegraphics[width=6cm]{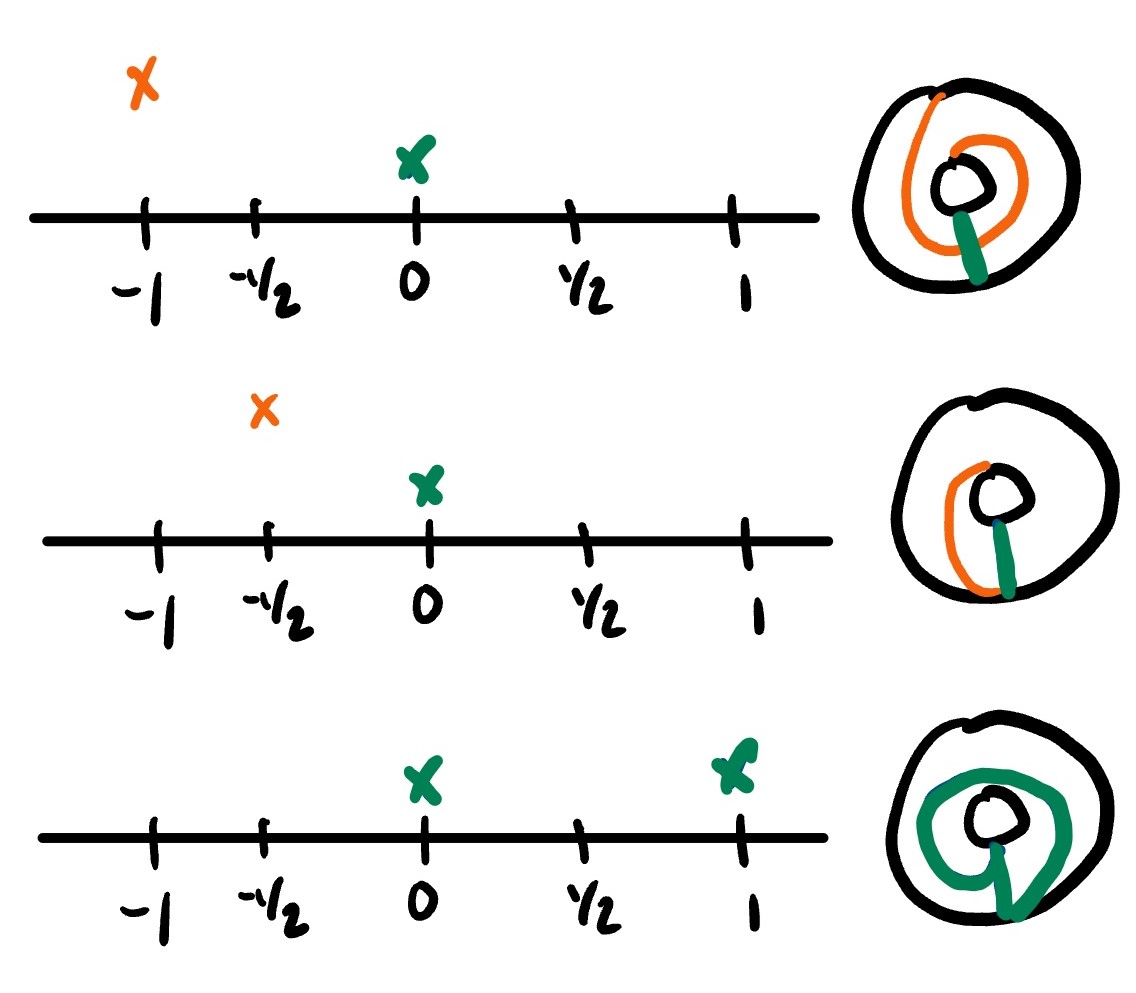}
    \caption{Examples of $\frac{1}{2} \Z_a$ (green) and $\frac{1}{2} \Z_d$ (orange) }
    \label{bookkeepex}
\end{figure}

We next introduce two copies of $\frac{1}{2} \Z$, labelled $\frac{1}{2} \Z_a$ and $\frac{1}{2} \Z_d$ corresponding via $\tau$ to $X_a^b\sqcup X_a^c$ and to $X_d^b\sqcup X_d^c$, respectively. Some examples are illustrated in Figure \ref{bookkeepex}. 

\begin{remark}\label{rem}
By Lemma \ref{mainth} for $\sigma_1,\sigma_2\in\frac{1}{2}\Z_a$, we have $|\tau(\sigma_1)-\tau(\sigma_2)|\leq 2$, so in particular there are most five arcs in $\frac{1}{2}\Z_a$. Furthermore, for $\sigma_1\in \frac{1}{2}\Z_a$ and $\sigma_2\in \frac{1}{2}\Z_d$, we have $|\tau(\sigma_1)-\tau(\sigma_2)|\leq \frac{3}{2}$
\end{remark}
Define the systems $\mathbb{F},\mathbb{G}$ as shown in Figure \ref{hehe}. Since $\deg(\mathbb{F})=(6,6,4,4)$ and $\deg(\mathbb{G})=(5,5,5,5)$, $\mathbb{F}\not\simeq\mathbb{G}$. By Remark \ref{rem}, we have the following:

\begin{figure}[htp]
    \centering
    \includegraphics[width=7cm]{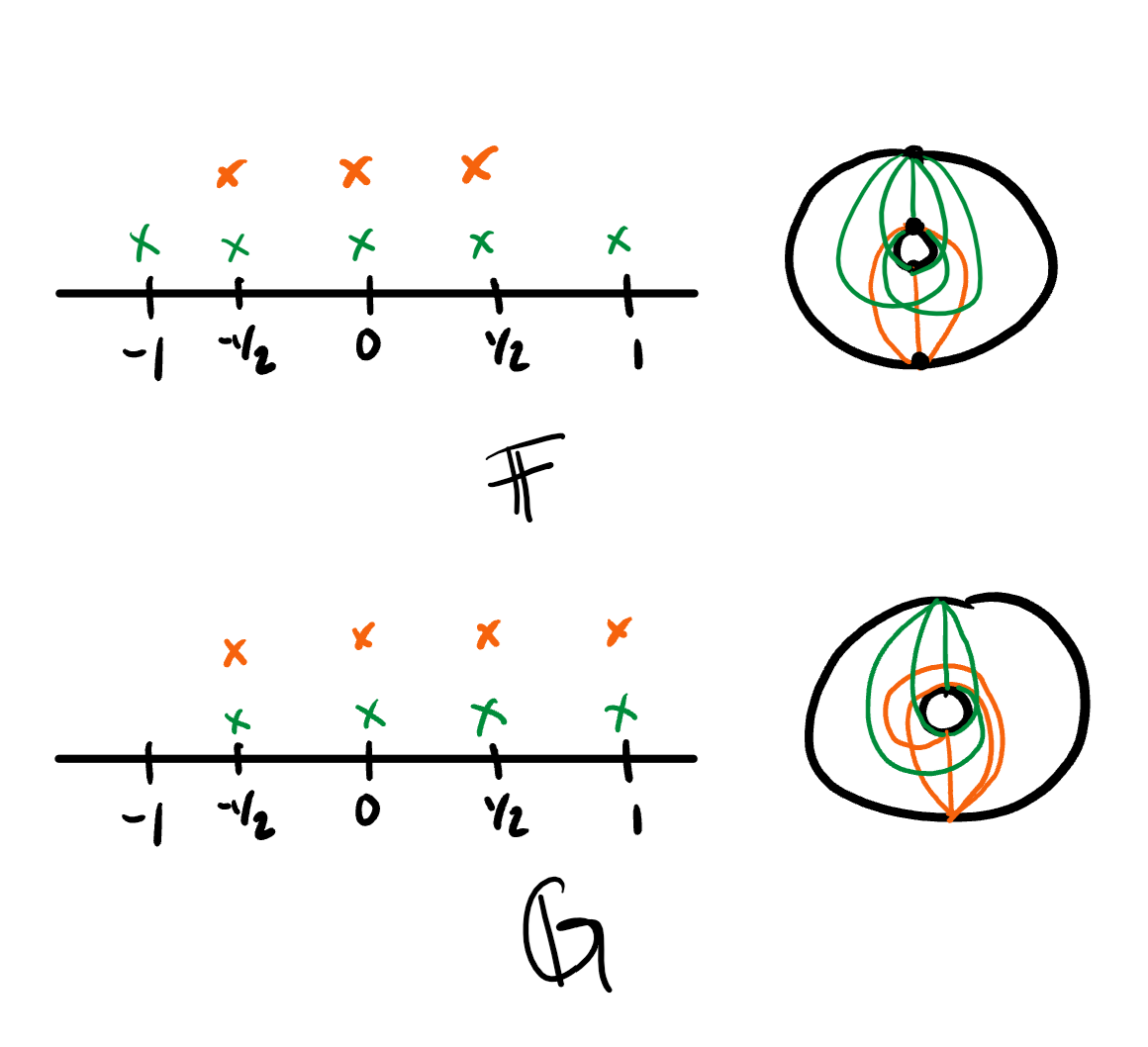}
    \caption{$\mathbb{F}, \mathbb{G}$ maximal 1-systems not containing loops}
    \label{hehe}
\end{figure}

\begin{corollary}\label{corann}
$\mathbb{F}$ and $\mathbb{G}$, as shown in Figure \ref{hehe}, are the two maximal 1-systems on $\Omega$ which do not contain loops.
\end{corollary}

The following proposition shows that any maximal 1-system on  $\Sigma$ must have nonseparating disjoint subset.

\begin{prop}\label{nonsep}
Let $\mathbb{A}$ be a maximal 1-system on $\Sigma$, in which all arcs are in pairwise minimal position. Then the disjoint subset $\mathbb{J}$ is nonseparating, that is, $\Sigma-\mathbb{J}$ is connected.

\begin{proof}
Suppose otherwise, so let $\mathbb{J}$ be the disjoint subset of maximal 1-system $\mathbb{A}$, and let $\mathbb{J}$ be separating. For each arc $\alpha\in \mathbb{A}-\mathbb{J}$, notice that $\alpha$ is contained in a connected component of $\Sigma-\mathbb{J}$. Let $\mathbb{C}\subset \Gamma_{\mathbb{J}}$ be a cycle. As we have four punctures, we have $|\mathbb{C}|\leq 4$. 

\begin{figure}[htp]
    \centering
    \includegraphics[width=3cm]{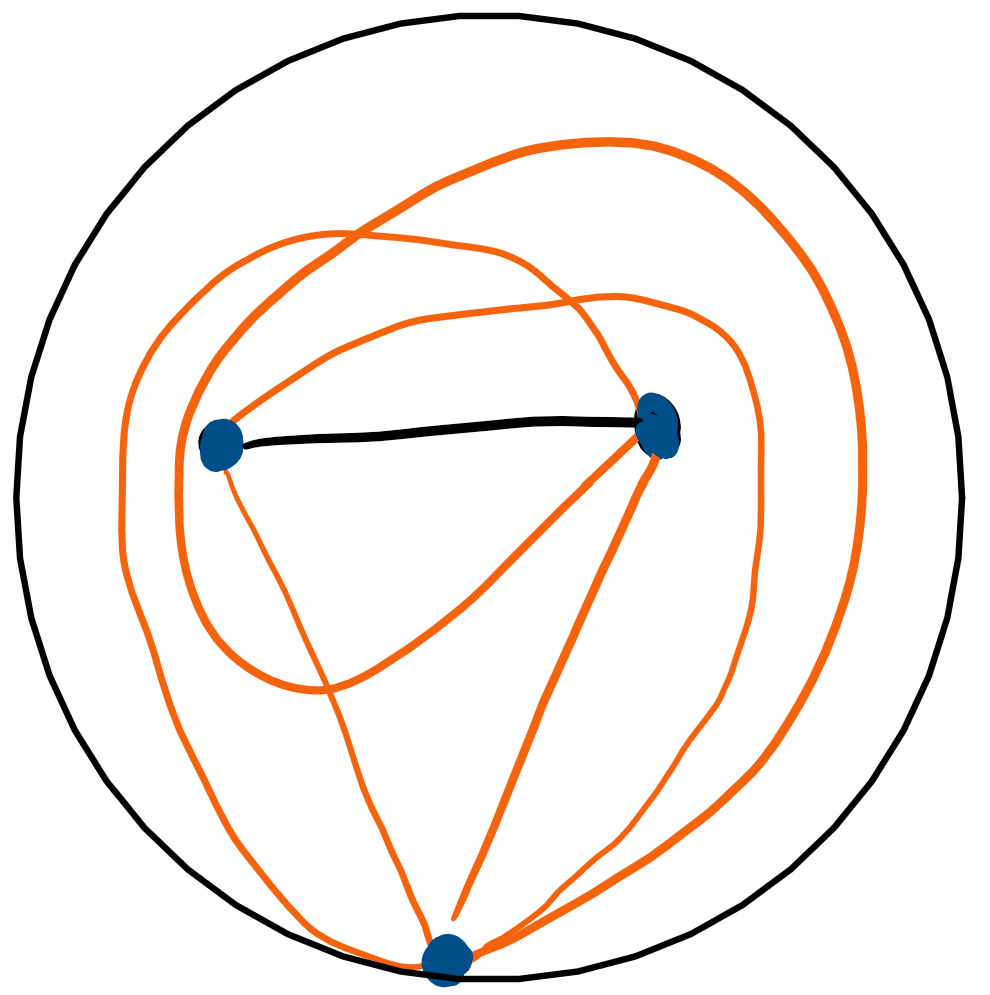}
    \caption{Cut $\mathcal{W}$ along $\sigma$ (black)}
    \label{nowayjose}
\end{figure}

First, assume $|\mathbb{C}|=1$, so $\mathbb{C}=\ell_p$, a loop based at $p$. $\ell_p$ divides $\Sigma$ into two components, one component with one puncture and the other component with two punctures. By Lemma \ref{tech}, the component with one puncture has at most 1 arc. For the other component $\mathcal{W}$, consider the unique class of arc between the two punctures. We can choose a representative $\sigma$ such that any arc  which is not a loop can be taken to miss this $\sigma$. Cutting along this arc, we are reduced down to Remark \ref{rem}, where we maximize the number of arcs between punctures $a,b,c$. Thus, we have $5$ arcs which are not loops in $\mathcal{W}-\sigma$, shown by the orange arcs in Figure \ref{nowayjose}. It is clear we can add at most 1 loop in $\mathcal{W}-\sigma$ based at the left puncture. This gives a total of at most $9$ arcs in $\mathbb{A}$, a contradiction.

If $|\mathbb{C}|=2$, then $\Sigma-\mathbb{C}$ must be two once-punctured bigons. By Lemma \ref{tech}, each punctured bigon has at most $3$ arcs. This gives a total of at most $8$ arcs in $\mathbb{A}$, a contradiction. 

If $|\mathbb{C}|=3$, then $\Sigma-\mathbb{C}$ must be a triangle and a once-punctured triangle. By Lemma \ref{tech}, the punctured triangle has at most $6$ arcs. This gives a total of at most $9$ arcs in $\mathbb{A}$, a contradiction.

If $|\mathbb{C}|=4$, then $\Sigma-\mathbb{C}$ must be two squares. As each square contains at most 2 arcs, we have at most 8 arcs if $\Gamma_\mathbb{J}$ contains a cycle of length $4$. This gives the desired contradiction. 

\end{proof}
\end{prop}

For the disjoint subset $\mathbb{J}\subset \mathbb{A}$, let $\Gamma_\mathbb{J}$ the induced embedded graph with vertex set $\mathcal{P}$ and edge set $\mathbb{J}$. Any loops or cycles in $\Gamma_\mathbb{J}$ will separate $\Sigma$. Combining with Lemma \ref{nonsep}, we see the following corollary. 

\begin{corollary}
$|\mathbb{J}|\leq 3$.
\end{corollary}

\section{Classifying Maximal 1-Systems}\label{1systems}
\subsection{\texorpdfstring{$|\mathbb{J}|=3$}{Lg}}\label{3cut}

\begin{prop}\label{classify3}
There are two maximal 1-systems on $\Sigma$ for $|\mathbb{J}|=3$, shown in Figure \ref{3cutmax}.
\end{prop}
As $\mathbb{J}$  is a tree, $\mathbb{J}$ must be either $T_a$ or $T_b$ from Section \ref{0systems}. In either case, $\Sigma-\Gamma_\mathbb{J}$ is a hexagon, and we need to find 9 arcs joining two non-adjacent vertices of a hexagon. Any two pairs of vertices must be used at most once by the homotopy condition. This forces us to take the complete graph on $6$ vertices. The two 1-systems, each corresponding to the following graphs in Figure \ref{3cutmax}. These form distinct 1-systems, as they have nonhomeomorphic $\Gamma_\mathbb{J}$.  

\begin{figure}[htp]
    \centering
    \includegraphics[width=8cm]{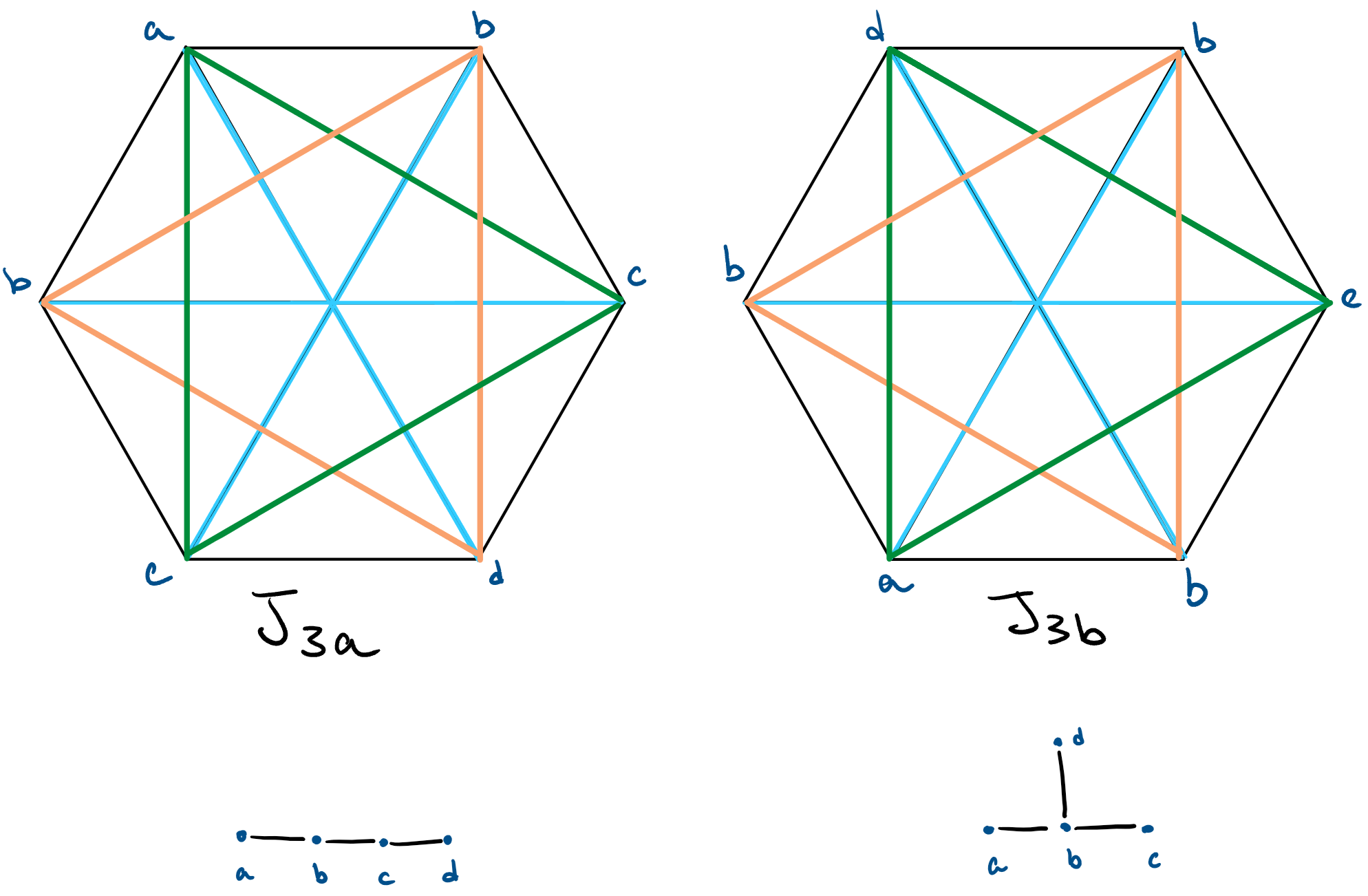}
    \caption{$J_{3a}, J_{3b}$ corresponding to $T_a$, $T_b$}
    \label{3cutmax}
\end{figure}

\subsection{\texorpdfstring{$|\mathbb{J}|=2$}{Lg}}\label{2cut}
\begin{prop}\label{classify2}
There are five maximal 1-systems on $\Sigma$ for $|\mathbb{J}|=2$. Two are shown in Figure \ref{max1syssq}, two are shown in Figure \ref{53loops}, and one is shown in Figure \ref{44loops}.
\end{prop}

If $\mathbb{J}$ contained 2 arcs, the distinguishing factor only is if the induced graph is connected or not. If so, then $\Sigma-\Gamma_\mathbb{J}$ is a once-punctured square. If not, then $\Sigma-\Gamma_\mathbb{J}$ is an annulus with two vertices and two edges for each boundary component. The two possibilities correspond to $P_4$ in Lemma \ref{tech}, and $\Omega$ in Corollary \ref{corann}. In both cases, we aim to find $10$ arcs in $\Sigma-\Gamma_\mathbb{J}$.

\begin{figure}[htp]
    \centering
    \includegraphics[width=8cm]{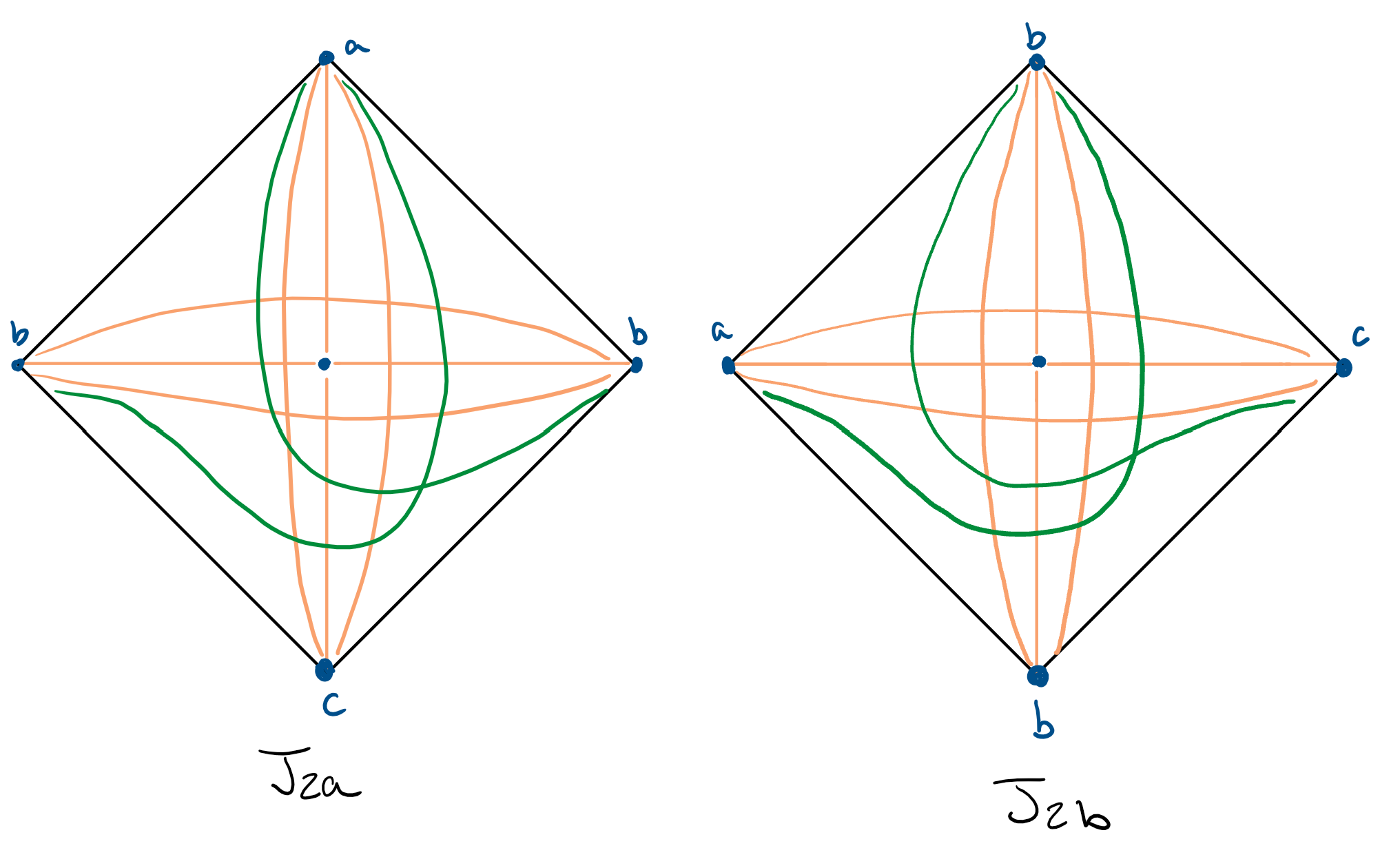}
    \caption{$J_{2a}, J_{2b}$ on $\mathcal{D}_4$ for connected $\Gamma_\mathbb{J}$}
    \label{max1syssq}
\end{figure}

Assume $\Gamma_\mathbb{J}$ is connected, and we see that we have two cases in Figure \ref{max1syssq}. By Lemma \ref{loopisajail}, we have no loops in $\mathbb{A}$. After adding in labels, We can distinguish the two systems by considering degrees: $\deg(J_{2a})=(6,10,4,4)$ and $\deg(J_{2b})=(5,10,5,4)$. Thus, for connected $\Gamma_\mathbb{J}$, Figure \ref{max1syssq} yields two maximal 1-systems.

Now assume $\Gamma_\mathbb{J}$ is disconnected. This brings us to the $\Omega$, and all that remains is to extend $\mathbb{F}$ and $\mathbb{G}$ in Corollary \ref{corann} to maximal 1-systems by adding in loops.

Adding loops to the $\mathbb{F}$ system yields Figure \ref{53loops}. All but two systems are distinguishable by degree considerations. The systems on the bottom left and top right are equivalent since one is obtained under an inversion in the core circle of $\mathbb{A}$. Thus, we have 3 distinct 1-systems corresponding to $\mathbb{F}$. 
\begin{figure}[htp]
    \centering
    \includegraphics[width=7cm]{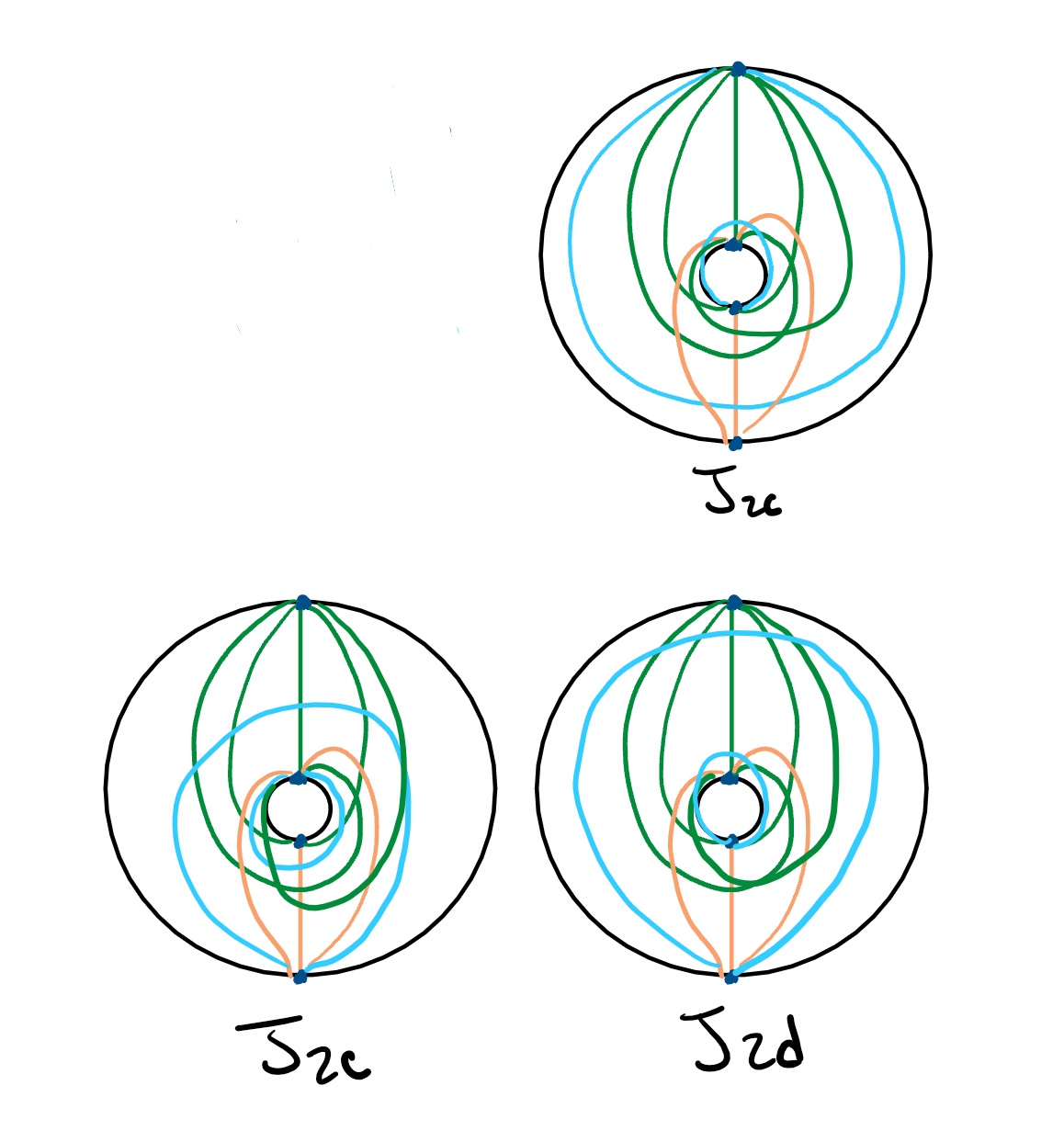}
    \caption{$J_{2c}, J_{2d}$ extended from $\mathbb{F}$}
    \label{53loops}
\end{figure}

Adding loops to the $\mathbb{G}$ system yields Figure \ref{44loops}. Note we cut down on half the cases using symmetry. Furthermore, the half Dehn-twist around the inside boundary circle of $\mathbb{A}$ maps the left figure to the right one reflected in the vertical symmetry axis of $\mathbb{A}$. This yields the unique maximal 1-system corresponding to~$\mathbb{G}$.
\begin{figure}[htp]
    \centering
    \includegraphics[width=6cm]{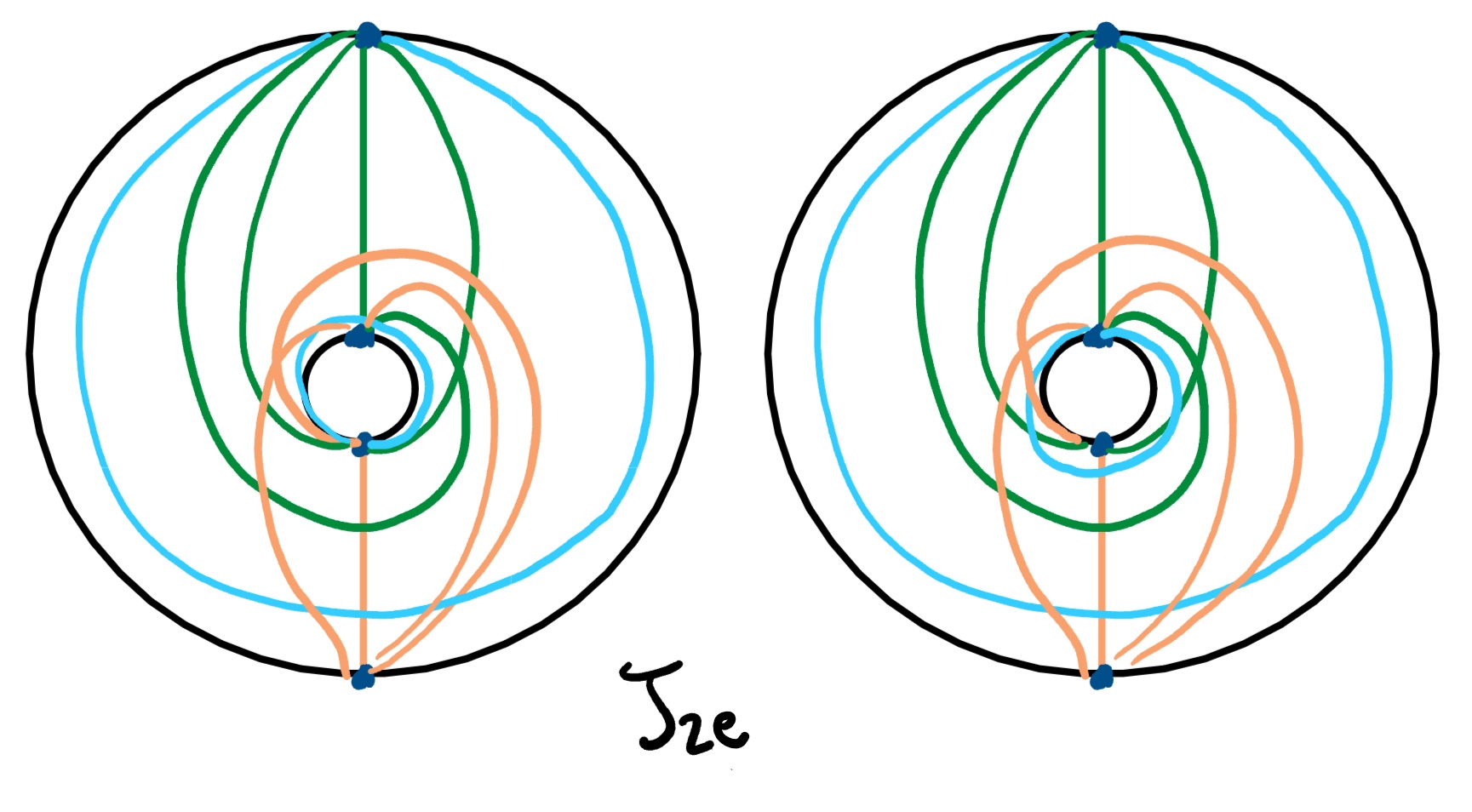}
    \caption{$J_{2e}$ extended from $\mathbb{G}$}
    \label{44loops}
\end{figure}

\subsection{\texorpdfstring{$|\mathbb{J}|=1$}{Lg}}\label{1cut}
\begin{prop}\label{clasisfy1}
There is one maximal 1-system on $\Sigma$ for $|\mathbb{J}|=1$, shown in Figure \ref{fig:ad}.
\end{prop}
Without loss of generality, let $\mathbb{J}$ be an arc between punctures $a,d$, so $\Sigma-\Gamma_\mathbb{J}$ is a twice punctured disk. We begin by discussing arcs which are not loops.

We begin by defining an invariant on homotopy classes of arcs between $a,d$ which we call $\tau^*$, and investigate how this relates to $\tau$. We state a characterization of arcs between $a,d$, following the classification of arcs on a three-punctured sphere.

\begin{remark}\label{halfdehntwists}
Let $\sigma$ be an arc between $a,d$, and $\mathcal{L}\subset \Sigma-\Gamma_{\mathbb{J}}$ a small disk surrounding the two punctures. Then, every homotopy class of arcs between $a,d$ is obtained by half-Dehn twists of $\sigma$ about $\partial \mathcal{L}$.
\end{remark}

\begin{figure}[htp]
    \centering
    \includegraphics[width=8cm]{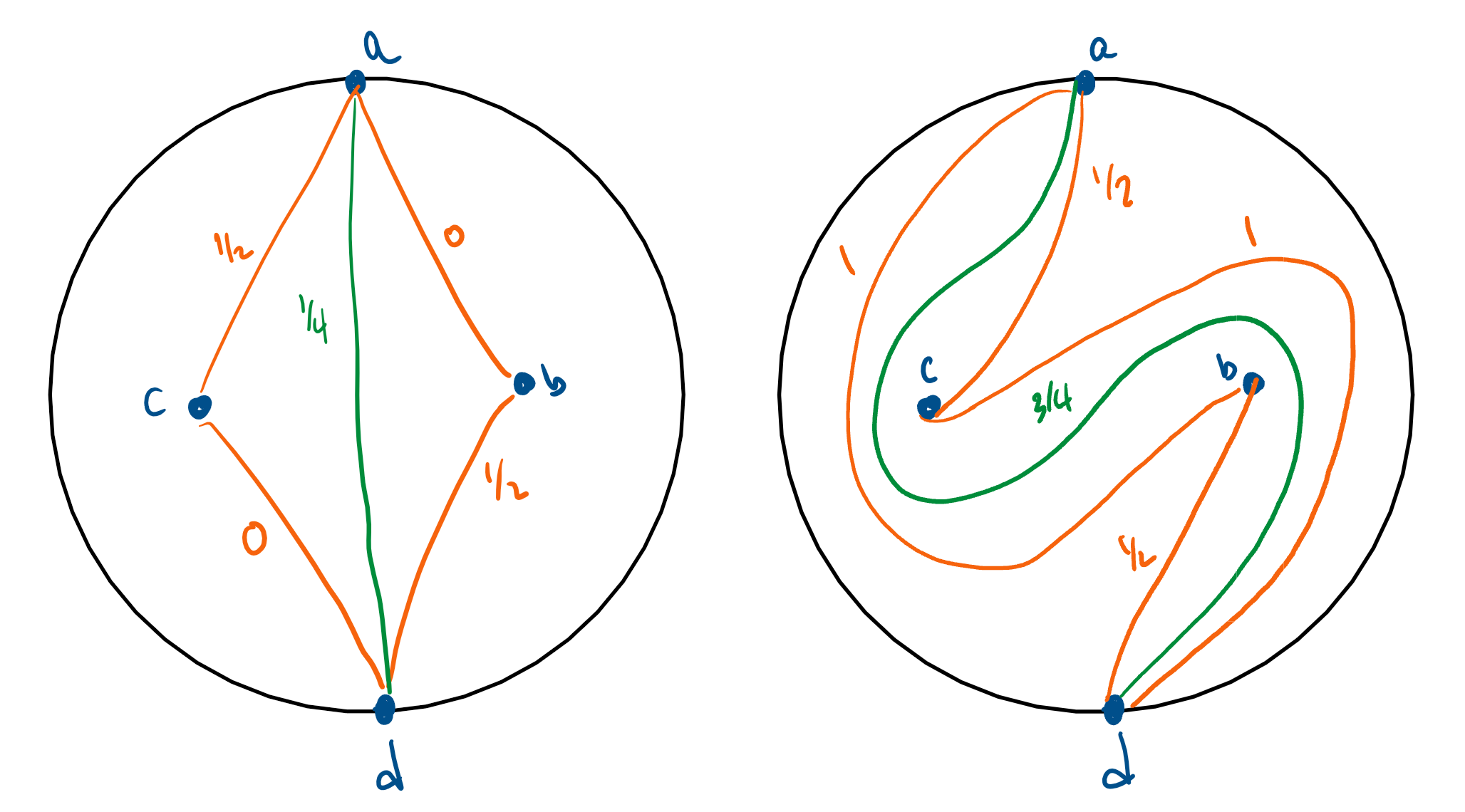}
    \caption{Examples of $\tau^*=\frac{1}{4}, \frac{3}{4}$}
    \label{halfdt}
\end{figure}

Observe each $\sigma$ between $a,d$ divides the $\Sigma-\Gamma_\mathbb{J}$ into two once-punctured bigons. Thus, we get four unique (homotopy classes of) arcs $\gamma_i$, where each arc has one endpoint in $\{a,d\}$ and the other in $\{b,c\}$, and is disjoint from $\sigma$. Figure \ref{halfdt} illustrates this where the $\sigma$'s are the green arcs, and the $\gamma_i$'s are the orange ones. Since there is a unique homotopy class of arcs between $b,c$, denoted $\phi$, every arc with an endpoint in $\{a,d\}$ and the other in $\{b,c\}$ can be characterized by its twisting number $\tau$ in $(\Sigma-\Gamma_{\mathcal{J}})-\phi$. This accounts for the orange numbers. Finally, we define $$\tau^*(\sigma)=\frac{1}{4}\sum_i \tau(\gamma_i)$$
as shown by the green numbers. Note that the image of $\tau^*$ is $\frac{1}{2}\Z+\frac{1}{4}$. In the same spirit as Lemma \ref{mainth}, by induction $\tau$ and $\tau^*$ completely determine the number of intersections of two arcs in minimal position. 

\begin{lemma}
If $\sigma_1,\sigma_2$ are two arcs between $a,d$, then
$$|\sigma_1\cap \sigma_2|=2|\tau^*(\sigma_1)-\tau^*(\sigma_2)|+1$$
\end{lemma}
We immediately have as a corollary an upper bound on the number of arcs between $a,d$.
\begin{corollary}\label{adad}
In any 1-system on $\Sigma-\Gamma_{\mathbb{J}}$, there at at most two arcs between $a,d$. 
\end{corollary}

\begin{lemma}\label{adbc}
If $\sigma$ is an arc between $a,d$ and $\gamma$ is an arc with one endpoint in $\{a,d\}$ and the other in $\{b,c\}$, then 
$$|\sigma\cap \gamma|=2|\tau^*(\sigma)-\tau(\gamma)|-\frac{1}{2}$$
\end{lemma}

We classify maximal 1-systems by cases using Corollary \ref{adad}. Suppose we had two arcs $\sigma_1,\sigma_2$ between $a,d$. Up to a homeomorphism, we can assume $\tau^*(\sigma_1)=\frac{1}{4}$ and $\tau^*(\sigma_2)=\frac{3}{4}$. By Lemma \ref{adbc}, the permissible values of $\tau$ are $\tau=0,\frac{1}{2}, 1$. We illustrate this in Figure \ref{fig:adad}.

\begin{figure}[htp]
    \centering
    \includegraphics[width=3cm]{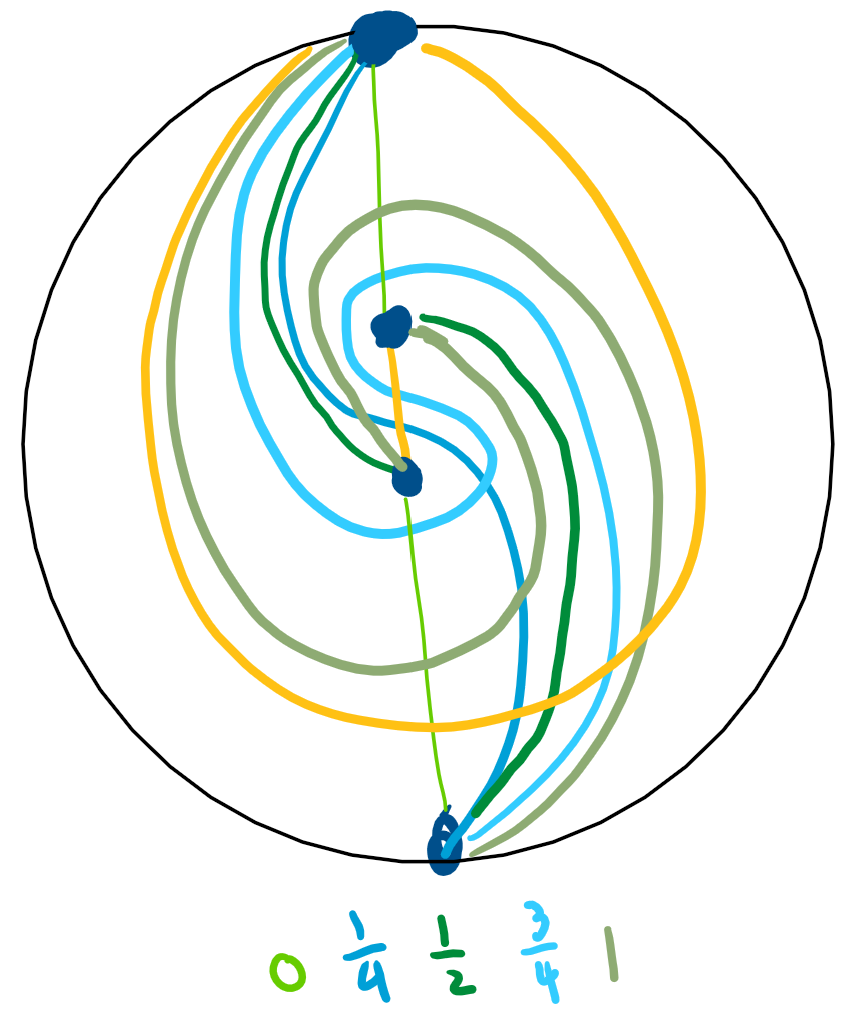}
    \caption{All permissible arcs given $\sigma_1,\sigma_2$}
    \label{fig:adad}
\end{figure}

This gives a maximum of: 2 arcs between $a,d$; 6 arcs between $\{a,d\}$ and $\{b,c\}$; $1$ arc between $b,c$; $1$ arc in $\mathbb{J}$. As we can have at most $1$ loop, this brings us to a maximum of 11 arcs. Thus, no such maximal 1-system can exist. 

Now assume there is one arc $\sigma$ between $a,d$. Up to a homeomorphism, we can assume $\tau^*(\sigma)=\frac{1}{4}$. By Lemma \ref{adbc}, the permissible values of $\tau$ are $\tau=0, \pm \frac{1}{2}, 1$. We illustrate this in Figure \ref{fig:ad}.

\begin{figure}[htp]
    \centering
    \includegraphics[width=3.5cm]{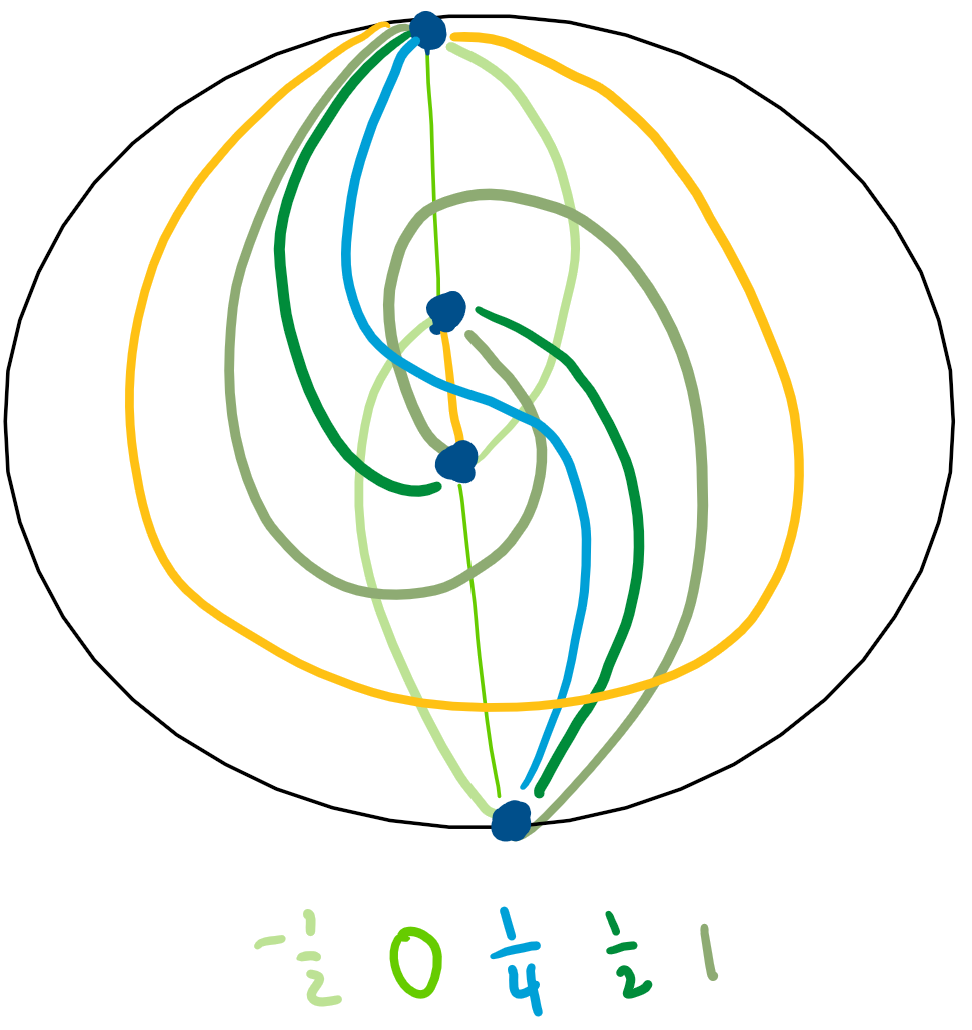}
    \caption{$J_1$ given $\sigma$ between $a,d$}
    \label{fig:ad}
\end{figure}

This gives a maximum of: 1 arc between $a,d$; 8 arcs between $\{a,d\}$ and $\{b,c\}$; $1$ arc between $b,c$; $1$ arc in $\mathbb{J}$. As we can have at most $1$ loop, this brings us to a maximum of 12 arcs. Adding in any loop, we get the maximal 1-system equivalent to Figure \ref{fig:ad}. 

Now assume there are no arcs between $a,d$. Then $\phi\in \mathbb{J}$, contradicting $|\mathbb{J}|=1$. Thus, no such maximal 1-system can exist.

\subsection{\texorpdfstring{$|\mathbb{J}|=0$}{Lg}}\label{0cut}
\begin{prop}\label{classify0}
There is one maximal 1-system on $\Sigma$ for $|\mathbb{J}|=0$, shown in Figure \ref{J0}.
\end{prop}
By Lemma \ref{loopisajail}, we observe that no loops can appear in the $|\mathbb{J}|=0$ case. Thus, any arcs mentioned will be understood to be arcs which are not loops. We begin by proving a slightly stronger statement of Theorem 1.7 of \cite{p2015} in the case of the four-punctured sphere. 

Let $\mathcal{P}$ denote the four punctures. Let $\mathcal{Q},\mathcal{Q}^*\subset \mathcal{P}$ each be a pair of distinct punctures such that $\mathcal{Q}\cap \mathcal{Q}^*=\emptyset$, and thus $\mathcal{Q}\cup \mathcal{Q}^*=\mathcal{P}$. We say $\mathcal{Q}$ and $\mathcal{Q}^*$ are a \textit{dual pair}. Note that $\mathcal{Q}$ completely determines $\mathcal{Q}^*$, and $(\mathcal{Q}^*)^*=\mathcal{Q}$

\begin{figure}[htp]
    \centering
    \includegraphics[width=3cm]{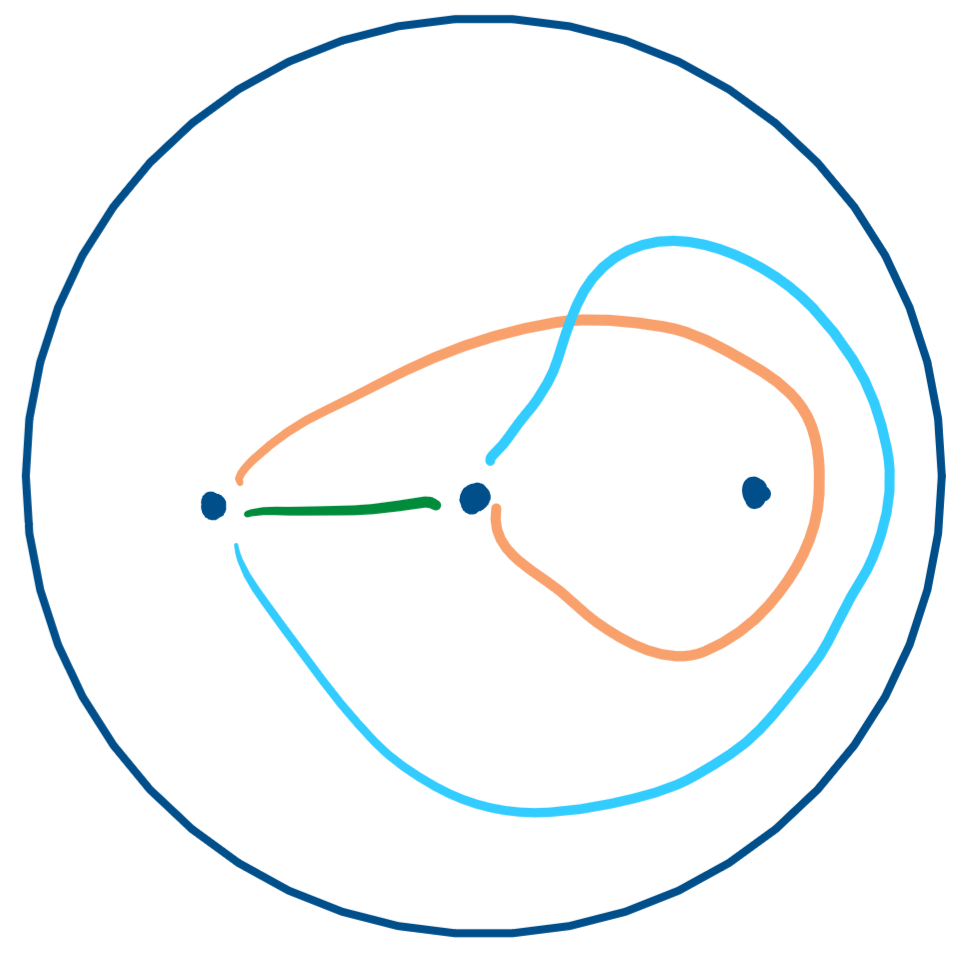}
    \caption{1-system of arcs between two punctures}
    \label{betweenpunc}
\end{figure}

\begin{lemma}\label{arcsbetween}
By \cite[Theorem 1.7]{p2015}, there are at most three arcs between $\mathcal{Q}$ on $\Sigma$. Furthermore, if three such arcs $\sigma_i$ are contained in a maximal 1-system $ \mathbb{A}$, then they are equivalent to Figure \ref{betweenpunc}, up to a homeomorphism of $\Sigma$.
\begin{proof}
It suffices to prove the second sentence. Up to relabelling, let us assume $\mathcal{Q}=\{a,b\}$. Assume by contradiction all $\sigma_i$ are disjoint. As the three arcs are nonhomotopic, 
$$\mathcal{Q}\cup \sigma_1\cup\sigma_2$$

bounds two once-punctured bigons, and
so it is homotopic to one of $\sigma_1,\sigma_2$, which is a contradiction. Then, as $\sigma_3$ is disjoint from both $\sigma_1,\sigma_2$, it must lie within one of the two bigons. But by Lemma \ref{tech}, this is impossible. 

\begin{figure}[htp]
    \centering
    \includegraphics[width=11cm]{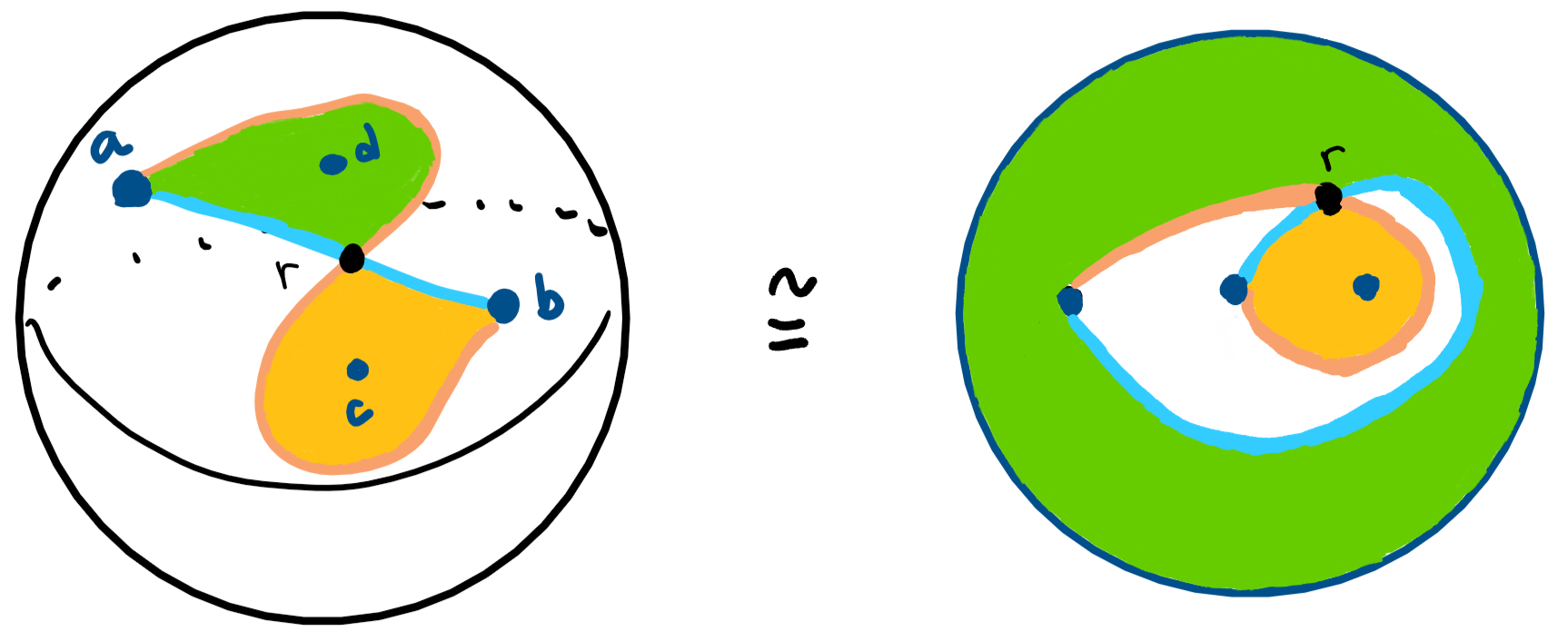}
    \caption{$\sigma_1$ (blue), $\sigma_2$ (orange) with intersection $r$}
    \label{sigmaqr}
\end{figure}

Let us assume $\sigma_1,\sigma_2$ intersect once, with intersection $r$. Up to a homeomorphism, the resulting configuration is illustrated in Figure \ref{sigmaqr},
where the colored regions must contain the remaining punctures by the bigon criterion. Thus, $\sigma_1,\sigma_2$ can be taken to be the orange and blue arcs on Figure \ref{betweenpunc}. 

Suppose by contradiction that $\sigma_3$ cannot be homotoped to the green arc in Figure \ref{betweenpunc}. Thus for some maximal 1-system $\mathbb{A}$ with $\sigma_1,\sigma_2\in \mathbb{A}$, we have some $\gamma\in \mathbb{A}$ with $1 < |\gamma\cap \sigma_3|$. See the first panel of Figure \ref{redarc} for this situation. 

\begin{figure}[htp]
    \centering
    \includegraphics[width= 8cm]{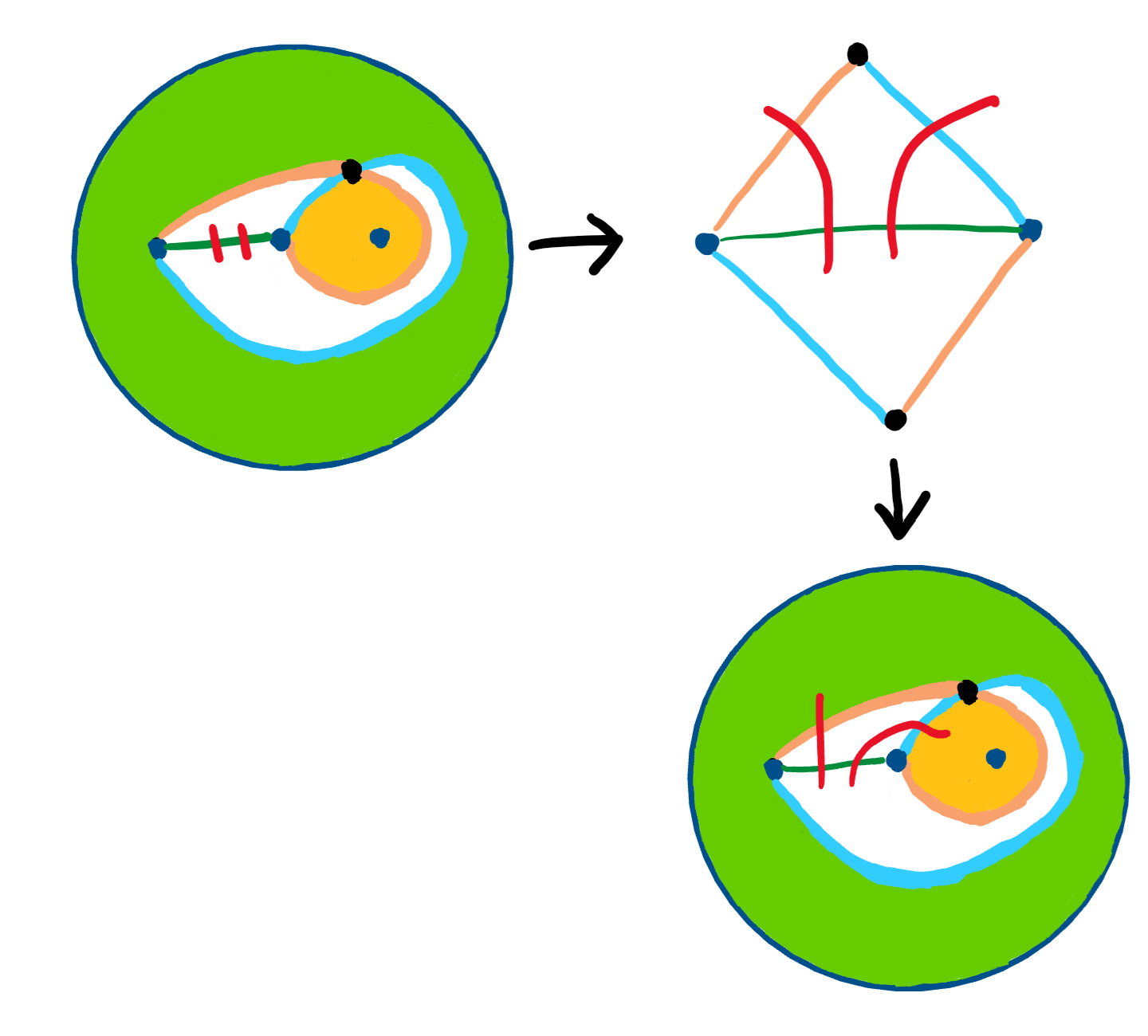}
    \caption{$\gamma$ in red intersecting $\sigma_3$ twice}
    \label{redarc}
\end{figure}

The top two strands of $\gamma$ must leave the white region. Thus, we have the second panel of Figure \ref{redarc}. However, at this stage we've intersected both $\sigma_1$ and $\sigma_2$ once. This implies the bottom strands must stay in the white region. But this implies $\gamma$ can be homotoped off $\sigma_3$, a contradiction.
\end{proof}
\end{lemma}

\begin{lemma}\label{abcdefg}
For any maximal 1-system $\mathbb{A}$ and dual pair $\mathcal{Q},\mathcal{Q}^*$, there are always exactly four arcs in $\mathbb{A}$ where each arc is between $\mathcal{Q}$ or $\mathcal{Q}^*$. 
\begin{proof}
We first show for a dual pair, there are no more than $4$ arcs between $\mathcal{Q}$ or $\mathcal{Q}^*$. Assume up to relabelling, $\mathcal{Q}=\{a,b\}$. If there are more than $4$ arcs, necessarily either $\mathcal{Q}$ or $\mathcal{Q}^*$ must contain $3$ arcs by Lemma \ref{arcsbetween}. Without loss of generality, assume $\mathcal{Q}$ contains $3$ arcs. Assume the two arcs which intersect once are the orange and blue arcs in Figure \ref{betweenpunc}. Thus, we have more than four arcs iff there exist more than one arc between $\mathcal{Q}^*$. There is one arc, namely the horizontal one, between $\mathcal{Q}^*$ which intersects all three regions, implying it intersects both colored arcs. Thus, any nonhomotopic arc must necessarily intersect the orange or blue arc once again. This gives an upper bound of $4$ arcs for each dual pair. But this in turn gives a lower bound of $4$ arcs, as we must have a total of $12$ arcs, $3$ choices of $(\mathcal{Q},\mathcal{Q}^*)$, and no loops. 
\end{proof}
\end{lemma}

\begin{corollary}\label{3122} For any maximal 1-system $\mathbb{A}$ and dual pair $\mathcal{Q},\mathcal{Q}^*$, exactly one of the two cases occurs:
\begin{enumerate}
    \item There are 3 arcs between $\mathcal{Q}$ and $1$ arc between $\mathcal{Q}^*$.
    \item There are 2 arcs between $\mathcal{Q}$ and $2$ arcs between $\mathcal{Q}^*$.
\end{enumerate}
\begin{proof}
Combine Lemmas \ref{arcsbetween} and \ref{abcdefg}.
\end{proof}
\end{corollary}

Our next goal is to show that the first case does not extend to a maximal 1-system, while the second case extends to a unique maximal 1-system. 

\begin{lemma}\label{no3}
Any 1-system with a $\mathcal{Q}$ such that 3 arcs are between $\mathcal{Q}$ does not extend to a maximal 1-system. 
\begin{proof}
Up to relabelling, let us assume $\mathcal{Q}=\{a,b\}$, and two of the intersecting arcs to be the orange and blue arcs as in Figure \ref{sigmaqr}. Taking $\Tilde{\mathcal{Q}}=\{a,d\}$, thus $\Tilde{\mathcal{Q}}^*=\{b,c\}$, we see that in neither case can we have more than one arc between $\Tilde{\mathcal{Q}}$ and $\Tilde{\mathcal{Q}}^*$. The only permissible arc between $\Tilde{\mathcal{Q}}$ and $\Tilde{\mathcal{Q}}^*$ lies entirely in the white and orange region respectively. Any other arc would lie in all three regions, which is disallowed by the pigeonhole principle and the definition of a 1-system. This contradicts Lemma \ref{3122}. 
\end{proof}
\end{lemma}

If fact, our proof shows a bit more. In the remaining case, where both $\mathcal{Q}$ and $\mathcal{Q}^*$ have two arcs between each, the two arcs between $\mathcal{Q}$ (and similarly $\mathcal{Q}^*$) must as in Figure \ref{22hello}.

\begin{figure}[htp]
    \centering
    \includegraphics[width=3cm]{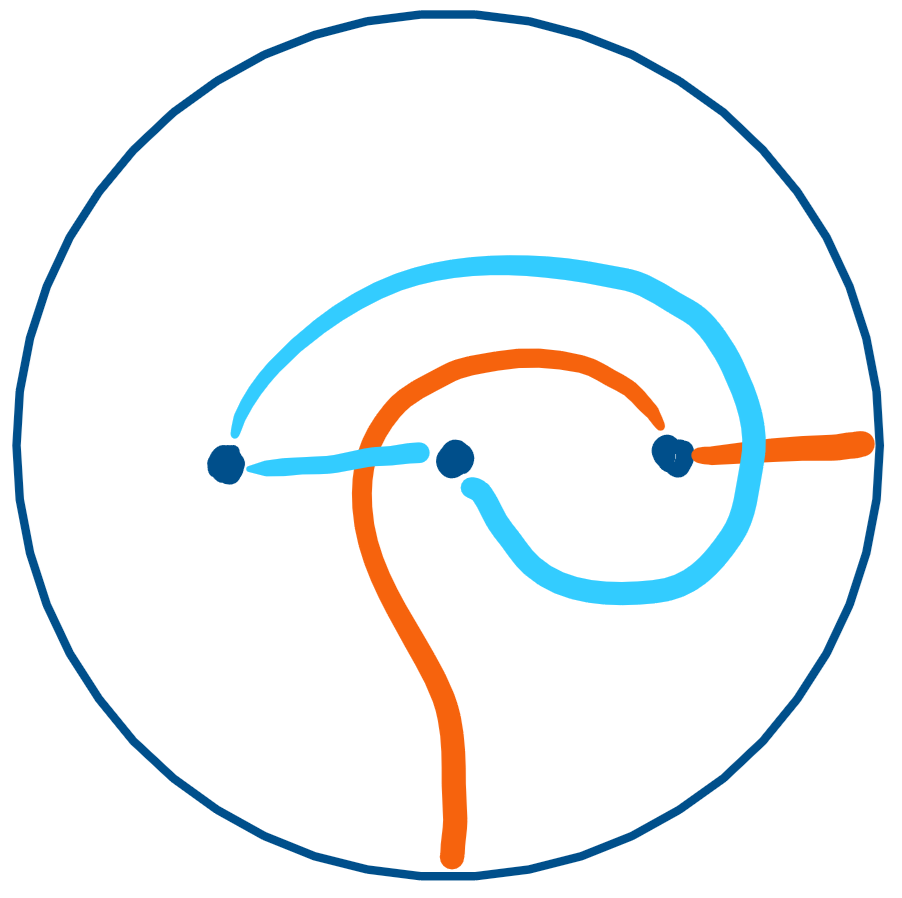}
    \caption{Two arcs between both $\mathcal{Q},\mathcal{Q}^*$}
    \label{22hello}
\end{figure}

\begin{lemma}\label{22looks}
Figure \ref{22hello} classifies the arcs between two punctures. 
\begin{proof}
By our above discussion, we know the blue (similarly orange) arcs must be disjoint. To see the blue arcs determine the orange arcs, it remains to argue each orange arc must intersect exactly one blue arc. Note that any orange arc must intersect at least one blue arc, to be between $\mathcal{Q}^*$. The orange arc intersects at most one blue arc, else we contradict the definition of a 1-system. 
\end{proof}
\end{lemma}

\begin{figure}[htp]
    \centering
    \includegraphics[width=5cm]{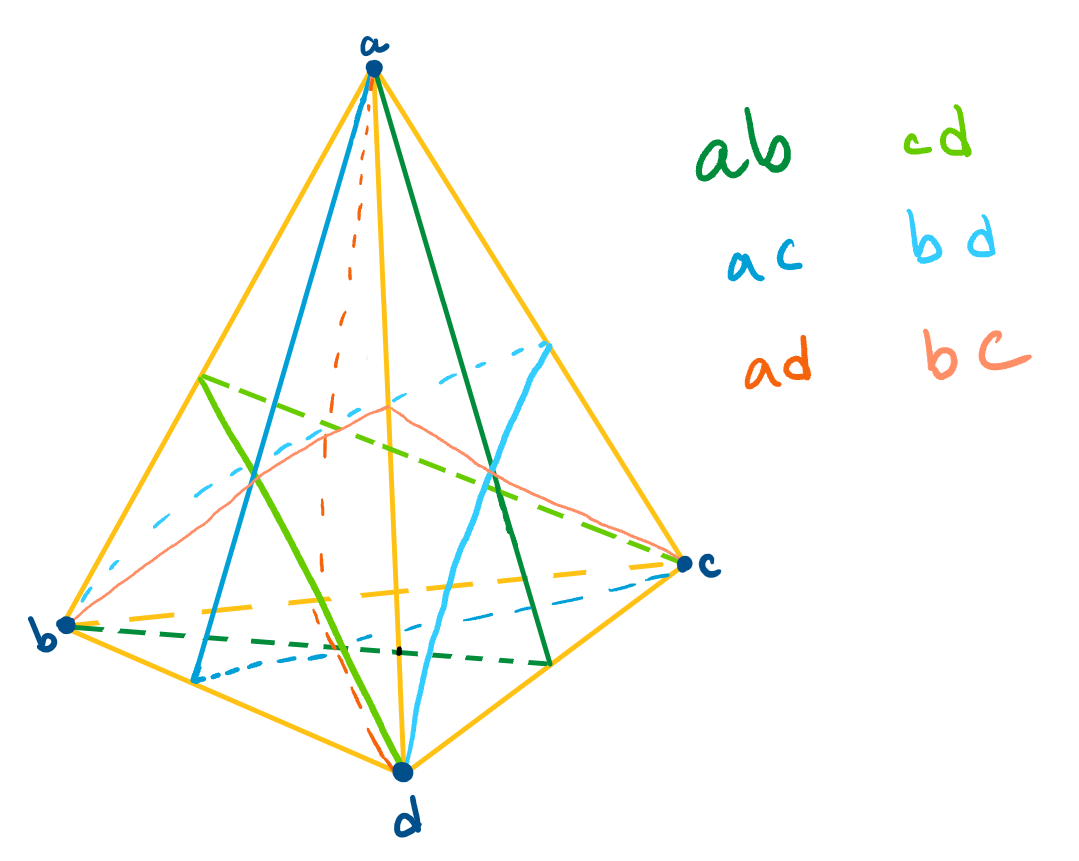}
    \caption{$J_0$ extended from pyramid}
    \label{J0}
\end{figure}

\begin{proof}[Proof of Proposition~\ref{classify0}]
We first prove that every maximal 1-systems contains the blue, orange and gray arcs in Figure \ref{fourregions}. By Lemma \ref{22looks}, we can take the blue and orange arcs to be in $\mathbb{A}$. 

\begin{figure}[htp]
    \centering
    \includegraphics[width=4cm]{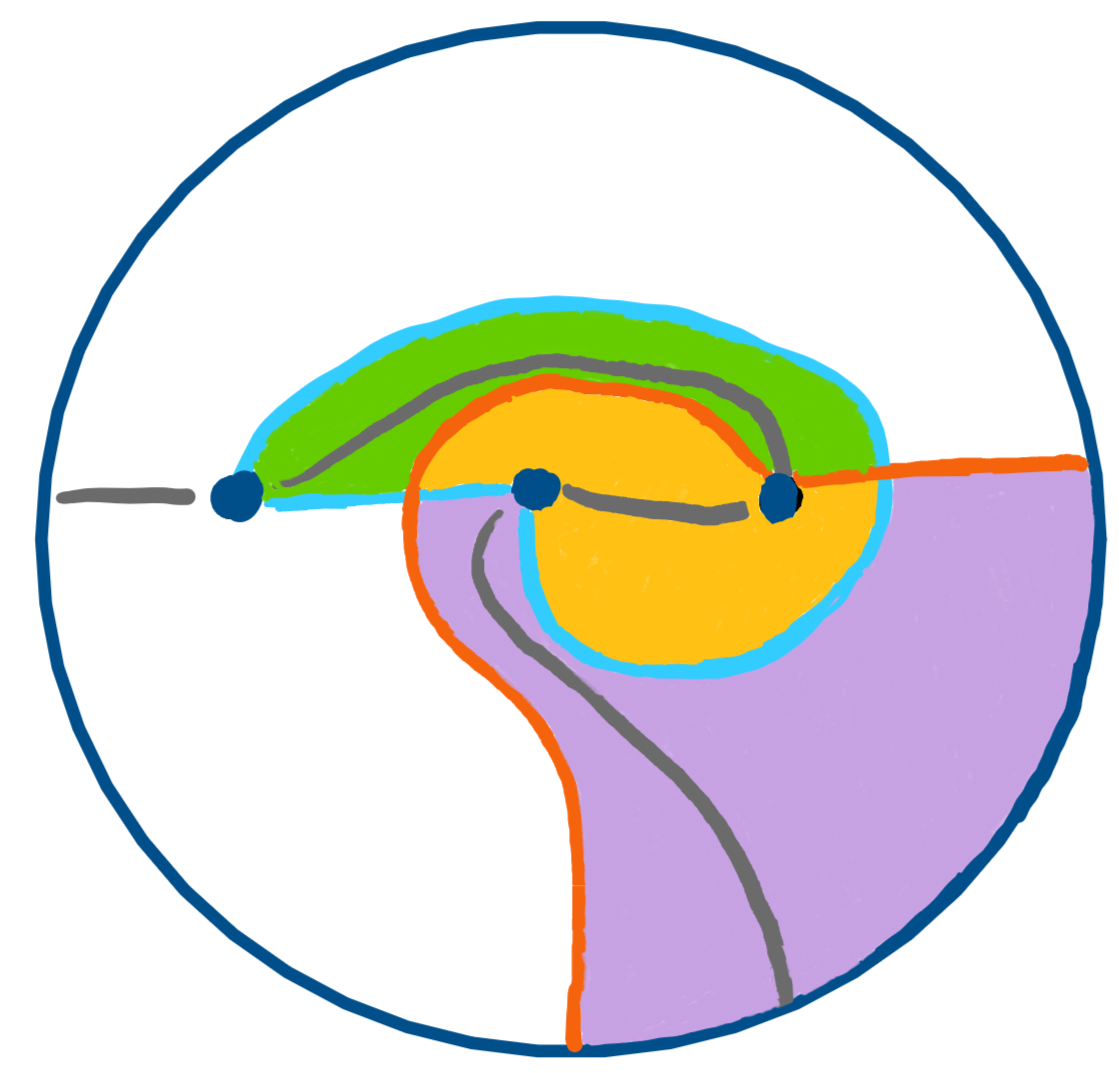}
    \caption{Four regions (white, green, yellow, purple) with gray arcs }
    \label{fourregions}
\end{figure}

The blue and orange arcs divide $\Sigma$ into four regions as in Figure \ref{fourregions}. Note also the four gray arcs which are entirely contained in a region. Let us show that the gray arc in the yellow region is contained in $\mathbb{A}$; the same argument will show the other gray arcs are also contained in $\mathbb{A}$. 

\begin{figure}[htp]
    \centering
    \includegraphics[width=8cm]{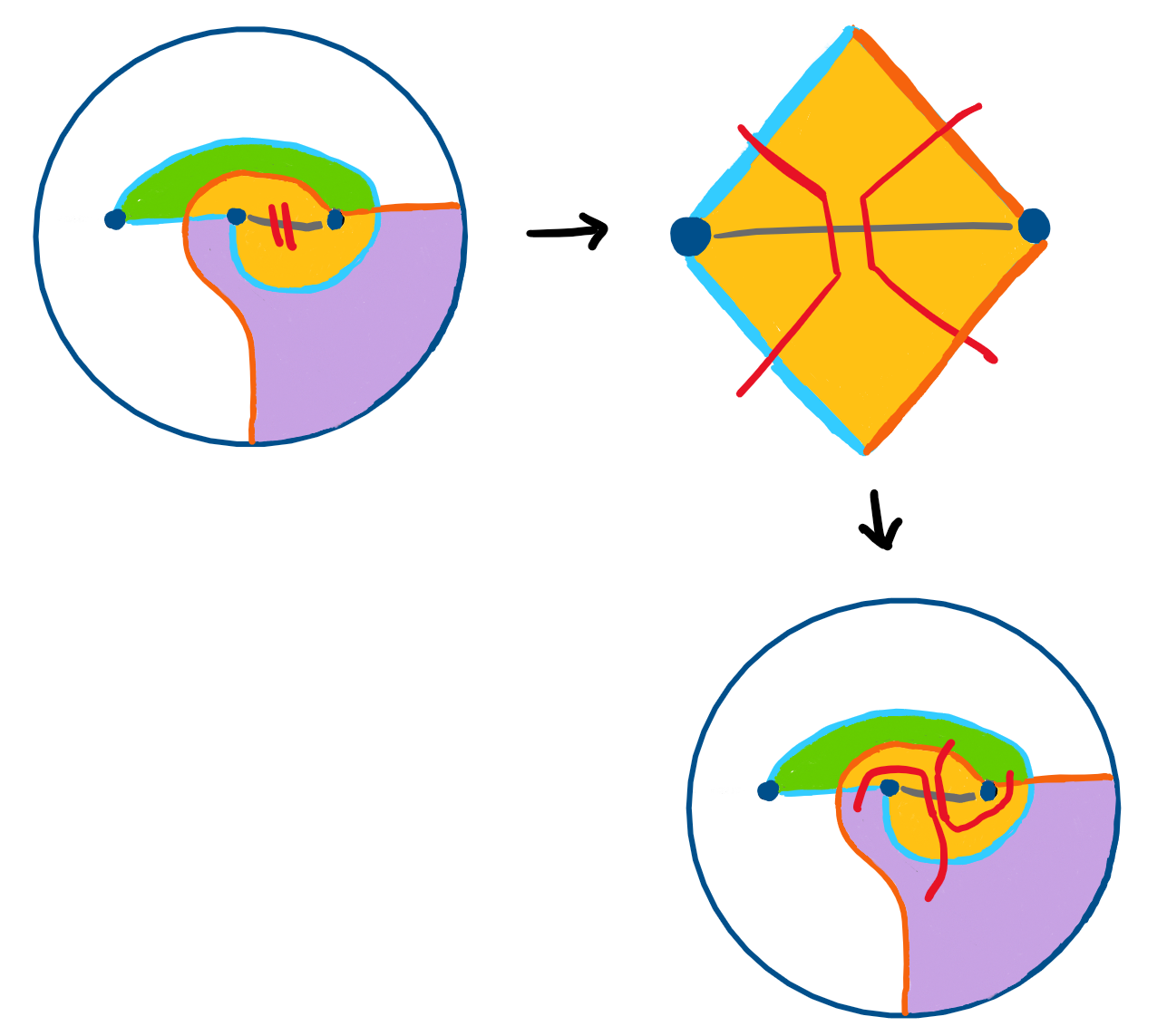}
    \caption{Gray arc must be contained in $\mathbb{A}$}
    \label{kkbonito}
\end{figure}

Suppose not, so for some maximal 1-system $\mathbb{A}$, there exists some arc which intersects the gray arc twice. This is illustrated in Figure \ref{kkbonito} by the red arc. The second panel ``zooms in" to show which curves the red arc must intersect by definition of 1-system. The last panel ``zooms out" to show the red arc cannot possibly close up, as it has already intersected with each of the four arcs once. This yields the desired contradiction. 

\begin{figure}[htp]
    \centering
    \includegraphics[width=9cm]{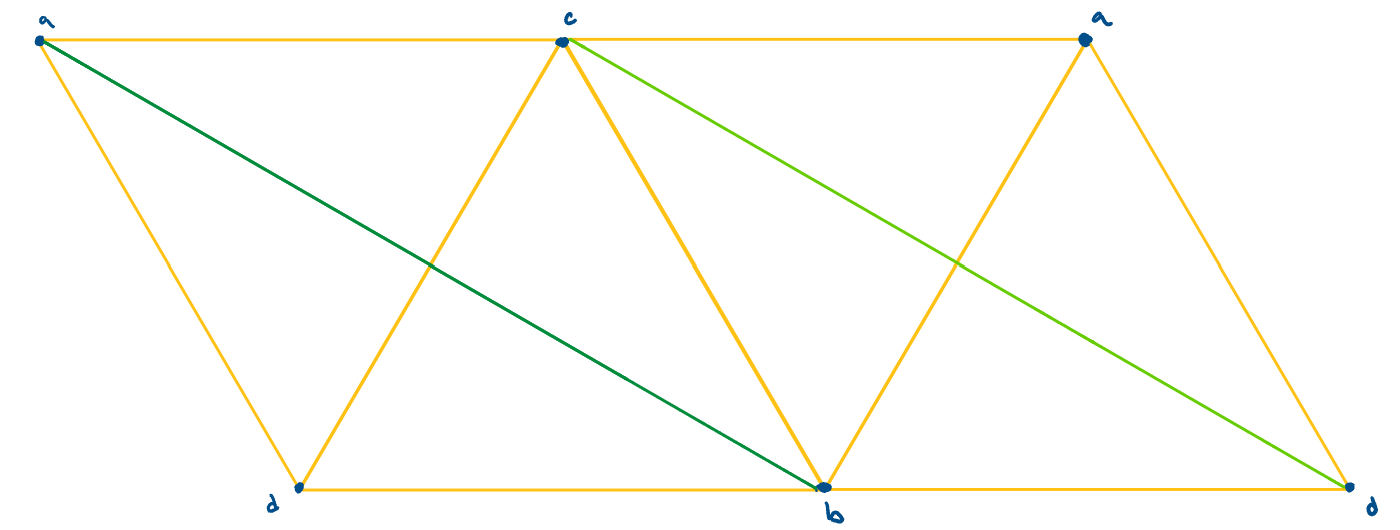}
    \caption{Arcs in Figure \ref{fourregions}}
    \label{unroll}
\end{figure}

Now that we know that $\mathbb{A}$ contains the arcs in Figure \ref{fourregions}, we will prove that $\mathbb{A}$ uniquely extends to Figure \ref{J0}. We begin by observing the arcs in Figure \ref{fourregions} are exactly the yellow arcs forming the pyramid and the two green arcs in Figure \ref{J0}. Thus, in the following, we say \textit{faces} to denote faces of the pyramid as in $\ref{J0}$. 

\begin{figure}[htp]
    \centering
    \includegraphics[width=8cm]{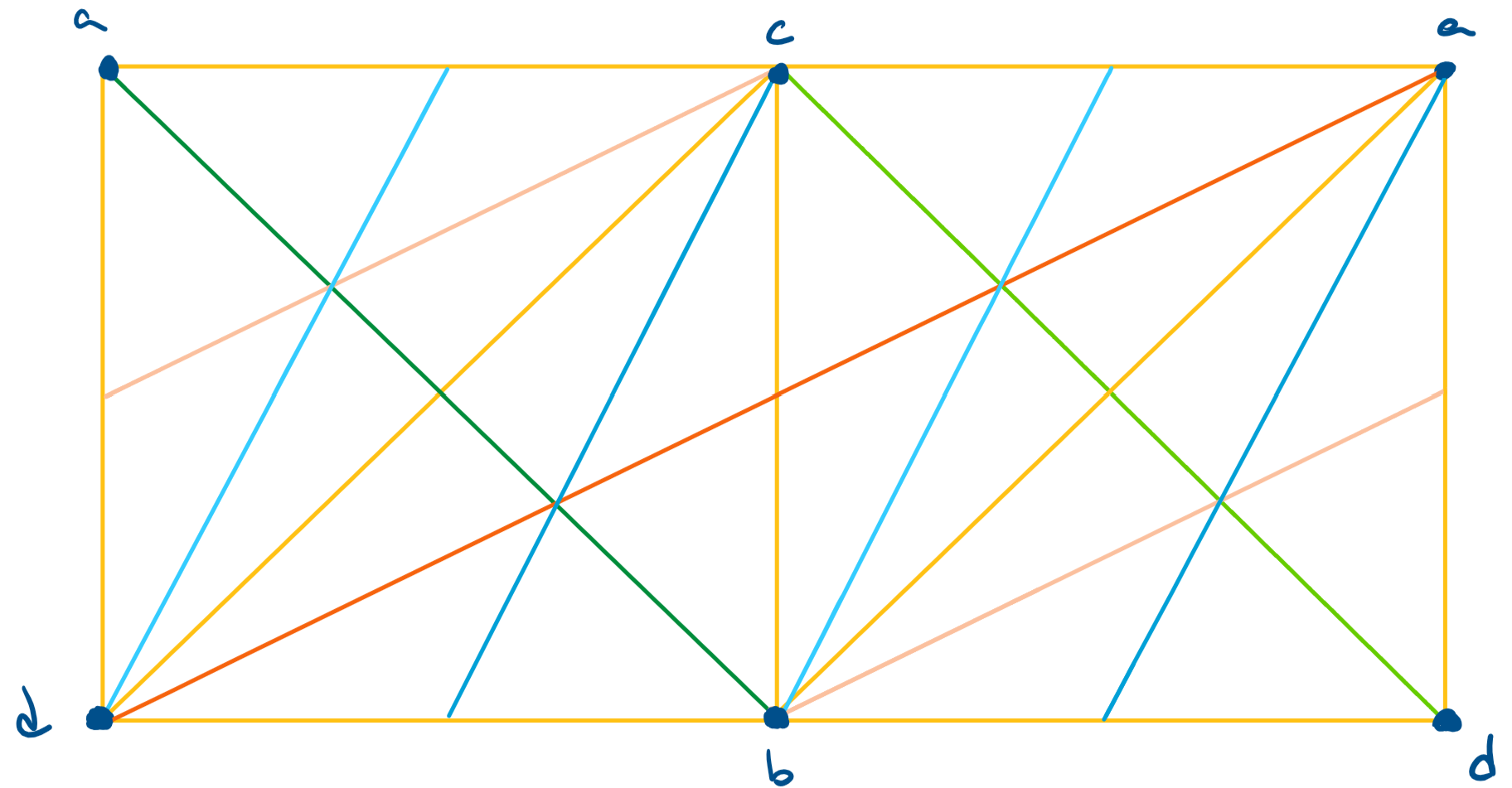}
    \caption{Unique Maximal 1-system $J_0$}
    \label{diagonaladbc}
\end{figure}

The arcs in Figure \ref{fourregions} unroll into Figure \ref{unroll}, a parallelogram. We will discuss,
given $\mathcal{Q}$, what can be the arc between  $\mathcal{Q}$ in $\mathbb{A}$ distinct from an edge of the pyramid which is guaranteed
by Lemma \ref{no3}. An arc through one face is homotopic to an arc in \ref{fourregions}, and through three implies the arc is a loop. Thus for $\mathcal{Q}=\{a,d\}$, it must be one of the two diagonals of the parallelogram. Tilting the parallelogram to a rectangle, we see these two choices are equivalent up to a  reflection through the vertical line through $b,c$. Without loss of generality, let us choose the one indicated by dark orange in Figure \ref{diagonaladbc}. But this choice determines an arc between $b,c$ by Lemma \ref{22looks} as indicated by the light orange arcs. It remains to choose an arc between $a,c$. Arcs through four faces either intersect the green or orange arc more than once. Thus, the only arc between $a,c$ is the one through two faces as seen in the blue arc in Figure \ref{diagonaladbc}. This uniquely determines the arc between $b,d$ indicated by light blue, and we see this is equivalent to Figure \ref{J0}.
\end{proof}

\printbibliography
\end{document}